\documentclass[a4paper, 10pt]{amsart}
\usepackage{amsmath}
\usepackage{amssymb, color}

\theoremstyle{plain}
\newtheorem{theorem}{\bf Theorem}[section]
\newtheorem{proposition}[theorem]{\bf Proposition}
\newtheorem{lemma}[theorem]{\bf Lemma}

\theoremstyle{definition}

\newtheorem{definition}[theorem]{\bf Definition}
\newtheorem{remark}[theorem]{\bf Remark}
\newtheorem{remarks}[theorem]{\bf Remarks}

\newcommand{\N}{\mathbb N}
\newcommand{\Z}{\mathbb Z}
\newcommand{\R}{\mathbb R}

 \DeclareMathOperator{\ord}{ord}
\DeclareMathOperator{\lcm}{lcm} 

 \DeclareMathOperator{\supp}{supp}

\newcommand{\red}{{\text{\rm red}}}
\renewcommand{\t}{\, | \,}

\numberwithin{equation}{section}

\begin{document}

\title{Local and global tameness in Krull monoids}

\address{Center for Combinatorics, Nankai University, Tianjin
300071, P.R. China} \email{wdgao\_1963@aliyun.com}

\address{Institut f\"ur Mathematik und Wissenschaftliches Rechnen\\
Karl--Fran\-zens--Universit\"at Graz\\
Heinrichstra{\ss}e 36\\
8010 Graz, Austria}

\email{alfred.geroldinger@uni-graz.at}

\address{Universit{\'e} Paris 13 \\ Sorbonne Paris Cit{\'e} \\ LAGA, CNRS, UMR 7539, Universit{\'e} Paris 8 \\ F-93430, Villetaneuse \\ France} \email{schmid@math.univ-paris13.fr}

\author{Weidong Gao, Alfred Geroldinger,  and Wolfgang A. Schmid}

\thanks{This work  was  supported  by NSFC,
Project Number 10971108,  by the Austrian Science Fund FWF, Project
Number P21576-N18,  by the Austrian-French Amad{\'e}e Program FR03/2012, and  by the ANR Project Caesar, Project Number ANR-12-BS01-0011.}



\begin{abstract}
Let $H$ be a Krull monoid with finite class group $G$ such that every class contains a prime divisor. Then the global tame degree $\mathsf t (H)$ equals zero if and only if $H$ is factorial
(equivalently, $|G|=1$). If $|G| > 1$, then $\mathsf D (G) \le \mathsf t (H) \le 1 + \mathsf D (G) \big( \mathsf D (G)-1 \big)/2$, where $\mathsf D (G)$ is the Davenport constant of $G$.
We analyze the case when $\mathsf t (H)$ equals the lower bound, and we show that $\mathsf t (H)$ grows asymptotically as the upper bound, when both terms are considered as functions of the rank of $G$. We provide more precise results if $G$ is either cyclic or an elementary $2$-group.
\end{abstract}

\maketitle

\medskip \centerline{\it Dedicated to Marco Fontana on the occasion of his 65th birthday} \medskip

%%%%%%%%%%%%%%%%%%%%%%%%%%%%%%%%%%%%%%%%%%%%%%%%%%%%%%%%%%%%%%%%%%%%%%%%%
%%                                      %%%%%%%%%%%%%%%
%%%%%%%%%%%%%%%%%%%%%%%%%%%%%%%%%%%%%%%%%%%%%%%%%%%%%%%%%%%%%%%%%%%%%%%%%

\section{Introduction} \label{1}

In an atomic monoid, every non-unit can be written as a finite product of atoms (irreducible elements). The multiplicative monoid of non-zero elements
from a noetherian domain is such an atomic monoid, and furthermore it is a Krull monoid if and only if the domain is integrally closed. In a given monoid $H$,
all factorizations into atoms are unique (in other words, $H$ is factorial) if and only if $H$ is a Krull monoid with trivial class group. Otherwise,
the non-uniqueness of factorizations is described by arithmetical invariants, such as sets of lengths, catenary and tame degrees.

The concepts of local and global tameness have found some attention in recent literature, and they were studied in settings ranging from numerical monoids to noetherian domains
(confer \cite{C-G-L-P-R06,Ge-Ha08a, C-G-L09, Ge-Ka10a,Ge-Gr-Sc-Sc10, Ph10a, Ph12b, Bl-Ga-Ge12a, Om12a, Ka11a, Ph12a}).
We recall their definitions.
Let $H$ be a monoid and $u \in H$  an atom. Then the local tame degree $\mathsf t (H, u)$ is the smallest $N$ with the following property: for any multiple $a$ of $u$ and any factorization $a = v_1 \cdot \ldots \cdot v_n$ of $a$, which does not contain the $u$, there is a short subproduct which is a multiple of $u$, say $v_1 \cdot \ldots \cdot v_m$, and a refactorization of this subproduct which contains $u$, say $v_1 \cdot \ldots \cdot v_m = u u_2 \cdot \ldots \cdot u_{\ell}$, such that $\max \{\ell, m \} \le N$. Thus the local tame degree $\mathsf t (H, u)$ measures the distance between any factorization of a multiple $a$ of $u$ and a factorization of $a$ which contains the $u$.
The global tame degree $\mathsf t (H)$ is the supremum of the local tame degrees over all atoms $u \in H$, and $H$ is called (globally) tame if the global tame degree $\mathsf t (H)$ is finite.

Local tameness is a basic finiteness property in the theory of non-unique factorizations in the sense that in many settings where an arithmetical finiteness property
has to be derived, local tameness has to be proved first (confer the proof of the Structure Theorem for sets of lengths in \cite[Section 4.3]{Ge-HK06a}). Krull monoids with finite class group  are globally tame. But if the domain fails to be integrally closed, this does not remain true any more, not even for non-principal orders in number fields. Indeed,  a non-principal order $\mathfrak o$ in an algebraic number field is always  locally tame, but  it is globally tame if and only if for every prime ideal $\mathfrak p$ containing the conductor there is precisely one prime ideal $\overline{ \mathfrak p}$ in the principal order $\overline{ \mathfrak o}$ such that $\overline{ \mathfrak p} \cap \mathfrak o = \mathfrak p$ (equivalently, if and only if its elasticity is finite).  Higher dimensional analogs  will be mentioned after Definition \ref{3.1}.

The focus of the present paper is on Krull monoids $H$ with finite class group $G$ such that every class contains a prime divisor, and for simplicity suppose now that $|G|>2$. There is the straightforward inequality
\[
\mathsf D (G) \le \mathsf t (H) \le 1 + \frac{\mathsf D (G)\big( \mathsf D (G)-1 \big)}{2} \,,
\]
where $\mathsf D (G)$ is the Davenport constant of $G$. We analyze the case when $\mathsf t (H)$ equals the lower bound, and we show that $\mathsf t (H)$ grows asymptotically as the upper bound, when both terms are considered as functions of the rank of $G$ (Theorem \ref{4.11}). This  result, which indicates the general behavior of the tame degree, will be complemented by
more precise results if $G$ is either cyclic or an elementary $2$-group (Theorems \ref{5.1} and \ref{5.2}). Arithmetical invariants (such as sets of lengths, sets of distances, the elasticity, the catenary degree, or the monotone catenary degree) of a Krull monoid as above depend only on the class group $G$ but not on the number of prime divisors in the classes, and therefore all investigations can be carried through in the associated monoid of zero-sum sequences instead of doing them in $H$.  In general, this is not the case for the tame degree, and we provide the first example revealing this fact (see Theorem \ref{5.1}, but also Proposition \ref{3.3}.1 and Remark \ref{3.4}.1).  Moreover, note that the existing
computational methods (as first presented in  \cite{C-G-L-P-R06}) cannot be applied to obtain this or  many other  examples given in the present paper
(the problem is the large number of variables involved in the systems of linear Diophantine equations to be solved).

\section{Preliminaries} \label{2}

Let $\mathbb N$ denote the set of positive integers, $\mathbb P \subset \mathbb N$ the set of prime numbers and $\N_0 = \N \cup \{0\}$. For real numbers $a, b \in \mathbb R$, let $[a, b] = \{ x \in \Z \mid a \le x \le b\}$ an interval of integers. By a monoid, we mean a commutative semigroup with unit element which satisfies the cancellation laws. All our concepts will be formulated in the language of monoids. The monoids we have in mind are  multiplicative monoids of nonzero elements of noetherian or Mori domains, monoids of ideals (with suitable multiplication), and additive monoids of certain classes of modules (\cite{HK98, Ge-HK06a,  Ba-Wi13a, Fo-Ho-Lu13a}).

\medskip
\noindent
{\bf Arithmetic of monoids.} Let $H$ be a monoid. We denote by $\mathsf q (H)$ a quotient group of $H$ with $H \subset \mathsf q (H)$, by $H^{\times}$ the group of invertible elements, and by $H_{\red} = \{ a H^{\times} \mid a \in H\}$ the associated reduced monoid. We say that $H$ is reduced if $H^{\times} = \{1\}$. Furthermore, let $\mathcal A (H)$ be the set of atoms (irreducible elements) of $H$.
For a set $P$, we denote by $\mathcal F (P)$ the free (abelian) monoid with basis $P$. Then every $a \in \mathcal F (P)=F$ has a unique representation in the form \[
a = \prod_{p \in P} p^{\mathsf v_p(a) } \quad \text{with} \quad
\mathsf v_p(a) \in \N_0 \ \text{ and } \ \mathsf v_p(a) = 0 \ \text{
for almost all } \ p \in P \,.
\]
We call $\supp_P (a) = \supp (a) = \{ p \in P \mid \mathsf v_p
(a)
> 0 \} \subset P$ the {\it support} of $a$, and $|a|_F = |a| = \sum_{p\in P} \mathsf v_p (a) \in \mathbb N_0$ the {\it length} of $a$.
We will often consider submonoids of free abelian monoids, and in all these situations the length $|\cdot|$  refers to the largest free abelian monoid under consideration.
The free monoid \ $\mathsf Z (H) = \mathcal F \bigl( \mathcal
A(H_\red)\bigr)$ \ is called the \ {\it factorization monoid} \ of
$H$, and the unique homomorphism
\[
\pi \colon \mathsf Z (H) \to H_{\red} \quad \text{satisfying} \quad
\pi (u) = u \quad \text{for all} \quad u \in \mathcal A(H_\red)
\]
is called the \ {\it factorization homomorphism} \ of $H$. For $a
\in H$, the set
\[
\begin{aligned}
\mathsf Z (a)  & = \pi^{-1} (aH^\times) \subset \mathsf Z (H) \quad
\text{is the \ {\it set of factorizations} \ of \ $a$}, \quad \text{and}
\\
\mathsf L (a) & = \bigl\{ |z| \, \bigm| \, z \in \mathsf Z (a)
\bigr\} \subset \N_0 \quad \text{is the \ {\it set of lengths} \ of
$a$}  \,.
\end{aligned}
\]
By definition, we have \ $\mathsf Z(a) = \{1\}$ and $\mathsf L (a) = \{0\}$ for all $a \in
H^\times$. The monoid $H$ is called
\begin{itemize}
\smallskip
\item {\it atomic} \  if  \ $\mathsf Z(a) \ne \emptyset$ \
      for all $a \in H$;

\smallskip
\item  {\it factorial} \  if  \ $|\mathsf Z (a)| =1$ \
       for all $a \in H$ (equivalently, $H$ is atomic and every atom is a
       prime).
\end{itemize}
Let $z,\, z' \in \mathsf Z (H)$. Then we can write
\[
z = u_1 \cdot \ldots \cdot u_{\ell}v_1 \cdot \ldots \cdot v_m \quad
\text{and} \quad z' = u_1 \cdot \ldots \cdot u_{\ell}w_1 \cdot \ldots
\cdot w_n\,,
\]
where  $\ell,\,m,\, n\in \N_0$,  $u_1, \ldots, u_{\ell},\,v_1, \ldots,v_m,\,
w_1, \ldots, w_n \in \mathcal A(H_\red)$ are such that
$
\{v_1 ,\ldots, v_m \} \cap \{w_1, \ldots, w_n \} = \emptyset$.
The {\it distance} between $z$ and $z'$ is defined by
\[
\mathsf d (z, z') = \max \{m,\, n\} = \max \{ |z \gcd (z, z')^{-1}|,
|z' \gcd (z, z')^{-1}| \} \in \N_0 \,.
\]

\medskip
\noindent
{\bf Krull monoids.}
A monoid homomorphism $\varphi \colon H \to D$   is called
\begin{itemize}
\item a  {\it divisor homomorphism} if $\varphi(a)\mid\varphi(b)$ implies that $a \t b$  for all $a,b \in H$.

\smallskip
\item  {\it cofinal}  if for every $a \in D$ there exists some $u
      \in H$ such that $a \t \varphi(u)$.

\smallskip
\item  a {\it divisor theory} (for $H$) if $D = \mathcal F (P)$
for some set $P$, $\varphi$ is a divisor homomorphism, and for every
$p \in P$ (equivalently for every $p \in \mathcal{F}(P)$), there exists a finite
subset $\emptyset \ne X \subset H$ satisfying $p = \gcd \bigl(
\varphi(X) \bigr)$.
\end{itemize}
The quotient group
$\mathcal{C}(\varphi)=\mathsf q (D)/ \mathsf q (\varphi(H))$ is called the
{\it class group} of $\varphi $.
For \ $a \in \mathsf q(D)$, we denote by \ $[a] = [a]_{\varphi} = a
\,\mathsf q(\varphi(H)) \in \mathsf q (D)/ \mathsf q (\varphi(H))$ \
the class containing \ $a$.  If $\varphi \colon H \to \mathcal F (P)$ is a
cofinal divisor homomorphism, then
\[
G_P = \big\{ [p] = p \, \mathsf q (\varphi(H)) \mid p \in P \big\} \subset
\mathcal{C}(\varphi)
\]
is called the \ {\it  set of classes containing prime divisors}.  By its very definition, every glass $g \in \mathcal{C}(\varphi)$ is a subset of $\mathsf q (D)$ and $P \cap g$ is the set of prime divisors lying in $g$.
The monoid $H$ is called a {\it Krull
monoid} if it satisfies one of the following equivalent properties
(\cite[Theorem 2.4.8]{Ge-HK06a} or \cite[Chapter 22]{HK98}){\rm \,:}
\begin{itemize}
\item[(a)] $H$ is $v$-noetherian and completely integrally closed,

\item[(b)] $H$ has a divisor theory,

\item[(c)] $H$ has a  divisor homomorphism into a free monoid.
\end{itemize}
If $H$ is a Krull monoid, then a divisor theory is essentially unique and the associated class group depends only on $H$ (it is called the class group of $H$).
An integral domain $R$ is a Krull domain if
and only if its multiplicative monoid $R \setminus \{0\}$ is a Krull
monoid, and thus Property (a) shows that a noetherian domain is
Krull if and only if it is integrally closed.

The main examples of Krull monoids which we have in mind are those stemming from number theory:
rings of integers in algebraic number fields, holomorphy rings in algebraic function fields and regular congruence monoids in these domains are Krull monoids with finite class group such that every class contains infinitely many prime divisors (\cite[Section 2.11]{Ge-HK06a}). If $R$ is an integral separable finitely generated algebra over an infinite field $k$
such that $\dim_k (R) \ge 2$, then $R$ is noetherian and every class contains infinitely many prime divisors (\cite{Ka99c}).
Monoid domains and power series domains that are Krull are discussed in \cite{Ki-Pa01, Ch11a}. For the role of Krull monoids in module theory we refer to
\cite{Fa06a, F-H-K-W06, Ba-Wi13a}. Module theory provides natural examples of Krull monoids where $G_P \subsetneq G$ but $G_P = -G_P$ holds true.

\medskip
\noindent
{\bf Monoids of zero-sum sequences.} Let $G$ be an additive abelian group,  $G_0
\subset G$ \ a subset and $\mathcal F (G_0)$ the free monoid with
basis $G_0$. According to the tradition of Combinatorial Number
Theory, the elements of $\mathcal F(G_0)$ are called \ {\it
sequences} over \ $G_0$. For a sequence
\[
S = g_1 \cdot \ldots \cdot g_{\ell} = \prod_{g \in G_0} g^{\mathsf v_g
(S)} \in \mathcal F (G_0) \,,
\]
we call
\[
\begin{aligned}
\sigma (S)  & = \sum_{i = 1}^{\ell} g_i \quad \text{the \ {\it sum} \ of \ $S$} \quad \text{and} \quad
\Sigma (S)  = \big\{ \sum_{i\in I} g_i \mid \emptyset \ne I \subset [1,\ell] \big\} \quad \text{the {\it set of subsums} of} \quad S
\,.
\end{aligned}
\]
Furthermore, $S$ is called {\it zero-sum free} if $0 \notin \Sigma (S)$, and it is a {\it minimal zero-sum sequence} if $|S| \ge 1$, $\sigma (S) = 0$ and $\sum_{i \in I} g_i \ne 0$ for all $\emptyset \ne I \subsetneq [1,\ell]$.
The monoid
\[
\mathcal B(G_0) = \{ U \in \mathcal F(G_0) \mid \sigma (U) =0\}
\]
is called the \ {\it monoid of zero-sum sequences} \ over \ $G_0$.
Since the embedding $\mathcal B (G_0) \hookrightarrow \mathcal F (G_0)$ is a divisor homomorphism,    $\mathcal B  (G_0)$ is a Krull monoid by Property (c).
The monoid $\mathcal B (G)$ is factorial if and only if $|G| \le 2$.
For every arithmetical invariant \ $*(H)$ \ defined for a monoid
$H$, it is usual to  write $*(G_0)$ instead of $*(\mathcal B(G_0))$
(whenever the meaning is clear from the context). In particular, we set \ $\mathcal A (G_0) = \mathcal
A (\mathcal B (G_0))$, $\mathsf Z (G_0) = \mathsf Z (\mathcal B (G_0))$,  and $\mathsf t (G_0) = \mathsf t (\mathcal B (G_0))$.
The atoms of $\mathcal B (G_0)$ are precisely the minimal zero-sum sequences over $G_0$, and
\[
\mathsf D (G_0) = \sup \{ |U| \mid U \in \mathcal A (G_0) \} \in \mathbb N \cup \{\infty\}
\]
is the {\it Davenport constant} of $G_0$.
Suppose that $G$ is finite with $|G| > 1$, say
\[
G \cong C_{n_1} \oplus \dots \oplus C_{n_{r}} \cong C_{q_1} \oplus \dots \oplus C_{q_{s}}\,,
\]
where $r, s \in \mathbb N$, $n_1, \ldots,n_{r} \in \N$, \ $1 < n_1 \t \ldots
\t n_{r}$, and $q_1, \ldots , q_{s}$ are
prime powers (not equal to $1$). Then
$r = \mathsf r (G)$ is the {\it rank} of $G$, $s = \mathsf r^* (G)$ is the {\it total rank} of $G$, $\mathsf d (G) = \mathsf D (G)-1$ is the maximal length of a zero-sum free sequence over $G$, and we define
\[
\mathsf d^* (G) = \sum_{i=1}^{\mathsf r (G)} (n_i - 1) , \ \mathsf D^* (G) = \mathsf d^* (G)+1 , \quad \text{and} \quad
\mathsf k^* (G) = \sum_{i=1}^{\mathsf r^* (G)} \frac{q_i - 1}{q_i} .
\]
Furthermore,  we set $\mathsf  d^*(\{0\}) = \mathsf k^* (\{0\}) = 0$. A straightforward example shows that $\mathsf D^* (G) \le \mathsf D (G)$. Moreover, equality holds for groups of rank $\mathsf r (G) \le 2$, for $p$-groups, and some other types of groups but not in general (\cite{Ge09a, Ge-Li-Ph12}).
If $t \in \N$ and $(e_1, \ldots, e_t) \in G^t$, then $(e_1, \ldots, e_t)$ is said to be {\it independent} if $e_1, \ldots, e_t$ are all nonzero and if, for every $(m_1, \ldots, m_t) \in \Z^t$, the equation $\sum_{i=1}^t m_i e_i = 0$ implies that $m_i e_i = 0$ for all $i \in [1, t]$. Furthermore, $(e_1, \ldots, e_t)$ is said to be a {\it basis} of $G$ if it is independent and $G = \langle e_1, \ldots, e_t \rangle$.

\section{Tameness and Transfer Homomorphisms} \label{3}

In this section we introduce the concepts of tameness and that of transfer homomorphisms. Our main reference is Section 3.2 in \cite{Ge-HK06a}.
We present the material in a way suitable for our applications in the following sections. Among others we will show that a Krull monoid is
locally tame if and only if the associated block monoid is locally tame, a fact which has not been observed so far. Furthermore, we establish a purely combinatorial characterization of the tame degree of a Krull monoid provided that every class contains sufficiently many prime divisors (Proposition \ref{3.5}).

\medskip
\begin{definition} \label{3.1}
Let $H$ be an atomic monoid.
\begin{enumerate}
\item For $b \in H$, let $\omega (H,b)$ denote the
      smallest $N\in \N_0 \cup \{\infty\}$ with the following property{\rm \,:}
      \begin{enumerate}

      \item[]
      For all $n \in \mathbb N$ and $a_1, \ldots, a_n \in H$, if $b \t a_1 \cdot \ldots \cdot a_n$, then there is a subset $\Omega \subset [1,n]$ such that
      \[
      |\Omega| \le N \quad \text{ and} \quad b \t \prod_{\nu \in \Omega} a_{\nu} \,.
      \]
      \end{enumerate}

\item For $a \in H$ and $x \in \mathsf Z (H)$, let $\mathsf t
      (a,x) \in \N_0 \cup \{\infty\}$ denote the
      smallest $N\in \N_0 \cup \{\infty\}$ with the following property{\rm \,:}
      \begin{enumerate}
      \smallskip
      \item[] If  $\mathsf Z(a) \cap x\mathsf Z(H) \ne \emptyset$ and
              $z \in \mathsf Z(a)$, then there exists $z' \in
              \mathsf Z(a) \cap x\mathsf Z(H)$ such that  $\mathsf d (z,z') \le
              N$.
      \end{enumerate}
      For  subsets $H' \subset H$ and $X \subset \mathsf Z(H)$, we
      define
      \[
      \mathsf t (H',X) = \sup \big\{ \mathsf t (a,x) \, \big| \, a \in H',  x \in
      X\big\} \in \N_0 \cup \{\infty\} \,.
      \]
      $H$ is said to be {\it locally tame} if $\mathsf t (H, u) <
      \infty$ for all $u \in \mathcal A (H_{\red})$.

\smallskip
\item  We set
       \[
       \omega (H) = \sup \{ \omega (H, u) \mid u \in \mathcal A (H) \} \in \mathbb N_0
      \cup \{\infty\} \,,
      \]
      and we call
      \[
      \mathsf t (H) = \mathsf t (H, \mathcal A \big(H_{\red})\big) = \sup \{
      \mathsf t (H, u) \mid u \in \mathcal A (H_{\red}) \} \in \mathbb N_0
      \cup \{\infty\}
      \]
      the {\it tame degree} of $H$. The monoid  $H$ is said to be {\it $($globally$)$ tame} if $\mathsf t (H) < \infty$.
\end{enumerate}
\end{definition}

\medskip
To analyze the above terminology, suppose that $H$ is reduced. By definition, an atom $u \in H$ is prime if and only if $\omega (H,u) = 1$.
Thus $\omega (H) = 1$ if and only if $H$ is factorial, and the $\omega (H, u)$ values measure how far away an atom is from being a prime.
Let $a \in H$ and $u \in \mathcal A (H)$. If $u \nmid a$, then $\mathsf t (a, u) = 0$. Otherwise, $\mathsf t (a, u)$ is the smallest $N$ with the following property:
if $z = v_1 \cdot \ldots \cdot v_n$ is any factorization of $a$ into atoms $v_1, \ldots, v_n$, then there is a subset $\Omega \subset [1,n]$, say $\Omega = [1,m]$, and
a factorization $z' = uu_2 \cdot \ldots \cdot u_{\ell}v_{m+1} \cdot \ldots \cdot v_n$ of $a$ with atoms $u_2, \ldots, u_{\ell}$ such that $\max \{\ell, m\} \le N$. If $u$ is a prime, then every factorization of $a$ contains $u$; thus we can choose $z'=z$ in the above definition, hence $\mathsf d (z,z') = 0$ and  $\mathsf t (H, u) = 0$.  If $u$ is not a prime, then $\omega (H, u) \le \mathsf t (H,u)$, and throughout this paper we will use the following characterization of $\mathsf t (H, u)$; it is the smallest  $N \in \N_0 \cup \{\infty\}$ with the following property:
\begin{enumerate}
\item[]
If \ $m \in \N$ and $v_1, \dots , v_m \in \mathcal A(H)$ are
such that \ $u \t v_1 \cdot \ldots \cdot v_m$, but $u$ divides no
proper subproduct of $v_1\cdot \ldots \cdot v_m$, then there exist
$\ell \in \N$ and  $u_2, \ldots , u_{\ell} \in \mathcal A
(H)$ such that $v_1 \cdot \ldots \cdot v_m = u u_2 \cdot \ldots
\cdot u_{\ell}$ and \ $\max \{\ell ,\, m \} \le N$ (in other words, $\max \{1+\min \mathsf L (u^{-1}v_1\cdot \ldots\cdot v_m), m\} \le N$).
\end{enumerate}
Globally, we have that  $H$ is factorial if and only if $\mathsf t (H) = 0$, and in the non-factorial case we have $\omega (H) \le \mathsf t (H)$. Moreover,  it is not difficult to show that $H$ is tame if and only if $\omega (H) < \infty$ (\cite{Ge-Ha08a}).

\smallskip
If $H$ is $v$-noetherian, then $\omega (H, b) < \infty $ for all $b \in H$, but this need not be true for the $\mathsf t (H, u)$ values. In other words, a $v$-noetherian monoid is not necessarily locally tame. Apart from Krull monoids which will be discussed below,  main examples of locally tame monoids are C-monoids: if $R$ is a noetherian domain with integral closure $\overline R$, non-zero conductor $\mathfrak f$, finite residue field $R/ \mathfrak f$ and finite class group $\mathcal C (\overline R)$, then $R$ is a C-monoid, and there is an explicit characterization when $R$ is globally tame
(see \cite[Theorem 2.11.9]{Ge-HK06a} and \cite{Ge-Ha08b, Ka11a, Re13a}).

A central method  to investigate arithmetical phenomena in a given class of monoids $H$ (such as noetherian domains)
is to construct a simpler auxiliary monoid $B$ and a homomorphism $\theta \colon H \to B$ which will be called a transfer homomorphism and which allows to shift arithmetical properties from $B$ to $H$. The machinery of transfer homomorphisms is most highly developed for Krull monoids but  not restricted to them.
The auxiliary monoids associated to Krull monoids are monoids of zero-sum sequences over their respective class groups. We start with the necessary definitions.

A monoid homomorphism \ $\theta \colon H \to B$ is called a  {\it
transfer homomorphism}  if  the following holds:

\smallskip

\begin{enumerate}
\item[]
\begin{enumerate}
\item[{\bf (T\,1)\,}] $B = \theta(H) B^\times$ \ and \ $\theta
^{-1} (B^\times) = H^\times$.

\smallskip

\item[{\bf (T\,2)\,}] If $u \in H$, \ $b,\,c \in B$ \ and \ $\theta
(u) = bc$, then there exist \ $v,\,w \in H$ \ such that \ $u = vw$,
\ $\theta (v) \simeq b$ \ and \ $\theta (w) \simeq c$.
\end{enumerate}
\end{enumerate}
A transfer homomorphism $\theta \colon H \to B$ between atomic monoids allows a unique extension $\overline \theta \colon \mathsf Z (H) \to \mathsf Z (B)$
to the factorization monoids satisfying $\overline \theta (uH^{\times}) = \theta (u)B^{\times}$ for all $u \in \mathcal A (H)$.

\medskip
For $a \in H$ and $x \in \mathsf Z (H)$, we denote by \ $\mathsf t
(a,x,\theta)$ \ the smallest $N \in \N_0 \cup \{\infty\}$ with the
following property:

\smallskip

\begin{enumerate}
\item[]
If \ $\mathsf Z (a) \cap x \mathsf Z(H) \ne \emptyset$, \ $z \in
\mathsf Z(a)$ and $\overline\theta (z) \in \overline \theta (x)
\mathsf Z(B)$, then there exists some $z' \in \mathsf Z (a) \cap x
\mathsf Z(H)$ such that $\overline \theta (z') = \overline \theta
(z)$ and \ $\mathsf d(z,z') \le N$.
\end{enumerate}

\smallskip\noindent
Then
\[
\mathsf t(H,x,\theta) = \sup \{\mathsf t(a,x,\theta) \mid a \in H \}
\in \N_0 \cup \{\infty\}
\]
is called the {\it tame degree in the fibres}. We will make substantial use of this concept in Section \ref{5}.

\medskip
\begin{lemma} \label{3.2}
Let $H$ be a reduced Krull monoid,  $H \hookrightarrow  F=\mathcal F (P)$ a cofinal divisor homomorphism, and let $G_P \subset G = F/H$ be the set of all classes containing
prime divisors.  Let $\widetilde{\boldsymbol \beta} \colon F \to \mathcal F
(G_P)$ denote the unique homomorphism defined by
$\widetilde{\boldsymbol \beta} (p) = [p]$ for all $p \in P$. Further, let $u \in \mathcal A (H)$ and $U = \widetilde{\boldsymbol \beta} (u)$.
\begin{enumerate}
\item The homomorphism $\boldsymbol \beta = \widetilde{\boldsymbol \beta}\mid H   \colon H \to \mathcal B
      (G_P)$ is a transfer homomorphism, and hence
      \[
      \mathsf t (G_P, U) \le \mathsf t (H, u) \le \mathsf t (G_P, U) + \mathsf t (H, u, \boldsymbol \beta) \,.
      \]

\smallskip
\item $\mathsf t(H, u, \boldsymbol \beta) \le 1 + |u| \le 1 + \mathsf D (G_P)$ \ for
        all $u \in \mathcal A(H)$.

\smallskip
\item Suppose that  $G_P = - G_P$ and that every nontrivial class contains at least two distinct prime divisors.
      Then $1 + |u| = \mathsf t(H, u, \boldsymbol \beta)$ \ for
        all $u \in \mathcal A(H)$ with $|u| \ge 3$. In particular,  if $\mathsf
        D (G_P) \ge 3$, then $
        \mathsf D (G_P) + 1 = \max \{ \mathsf t(H, u, \boldsymbol \beta) \mid u \in
        \mathcal A (H) \}$.
\end{enumerate}
\end{lemma}

\begin{proof}
See \cite[Theorem 3.2.5 and Proposition 3.4.8]{Ge-HK06a} for 1., and
\cite[Proposition 4.2]{Ge-Gr09b} for 2. and 3.
\end{proof}

For any Krull monoid $H$, we denote by $\widetilde{\boldsymbol \beta} \colon F \to \mathcal F
(G_P)$, by $\boldsymbol \beta \colon H \to \mathcal B (G_P)$, and by $\overline{\boldsymbol \beta} \colon \mathsf Z (H) \to \mathsf Z (G_P)$ the homomorphisms as defined above.

\medskip
\begin{proposition} \label{3.3}
Let $H$ be a reduced Krull monoid, $H \hookrightarrow F= \mathcal F(P)$ a cofinal divisor homomorphism,  $G
= F/H$, and $G_P \subset G$ the set of all classes containing
prime divisors such that that $\mathsf D (G_P) \ge 2$.
\begin{enumerate}
\item  $H$ is locally tame if and only if $\mathcal B
       (G_P)$ is locally tame. More precisely, we have
       \[
       \mathsf{t}(G_P, U) \le
       \mathsf{t}(H, u) \le \mathsf{t} (G_P, U) +  |U| + 1 \quad \text{for every} \quad u \in \mathcal A (H) \,.
       \]

\smallskip
\item $\mathsf t (H, u) \le  \max \left\{ \mathsf t (G_P, U), \frac{3+(|u|-1)(\mathsf D (G_P)-1)}{2}  \right\} \le 1 + \frac{|u|(\mathsf D (G_P)-1)}{2}$ \ for every $u \in \mathcal A (H)$.

\smallskip
\item $\mathsf t (G_P) \le \mathsf t (H) \le  \max \left\{ \mathsf t (G_P), \frac{3+(\mathsf D (G_P)-1)^2}{2}  \right\}  \le 1 + \frac{\mathsf D (G_P)(\mathsf D (G_P)-1)}{2}$, and if $\mathsf t (H) > \mathsf D (G_P)$, then
    \[
    \mathsf t (H) \le \max \{ 1+ \min \mathsf L (W) \mid W \in \mathcal B (G_P\setminus \{0\}), |W| \le \mathsf D (G_P) \big( \mathsf D (G_P)-1 \big) \} \,.
    \]

\smallskip
\item Suppose that  $G_P = -G_P$. If $U \in \mathcal A (G_P)$ with $|U| \ge 3$, then $\mathsf t (G_P, U) \ge |U|$. In particular, if  $\mathsf D (G_P) \ge 3$, then $\mathsf t (G_P) \ge \mathsf D (G_P)$.
\end{enumerate}
\end{proposition}

\begin{proof}
1. The inequalities follow immediately from Items 1. and 2. of Lemma \ref{3.2}, and by the very definition the inequalities show
that $H$ is locally tame if and only if $\mathcal B (G_P)$ is locally tame.

\smallskip
2.  See \cite[Theorem 3.4.10.6]{Ge-HK06a}.

\smallskip
3. The first inequalities  follow immediately from 1. and 2. Suppose that $\mathsf t (H) > \mathsf D (G_P)$.
If $\mathsf t (H)$ is infinite, then $\mathsf D (G_P)$ is necessarily also infinite and the last inequality is clearly true. So, suppose $\mathsf t (H)$ is finite. There are atoms $u, u_2, \ldots u_{\ell}, v_1, \ldots, v_m$ such that $u \t v_1 \cdot \ldots \cdot v_m$, but $u$ divides no proper  subproduct, and $\mathsf t (H) = \mathsf t (H, u) = \max \{\ell, m\}$. Since $m \le |u| \le \mathsf D (G_P)$, it follows that $\mathsf t (H, u) = \ell = 1 + \min \mathsf L (w)$ with $w = u^{-1} v_1 \cdot \ldots \cdot v_m$. Since
\[
|w| = |v_1 \cdot \ldots \cdot v_m| - |u| \le |u| \mathsf D (G_P) - |u| \le  \mathsf D (G_P) \big( \mathsf D (G_P)-1 \big) \,,
\]
the assertion follows.

\smallskip
4. Suppose that $U = g_1 \cdot \ldots \cdot g_m$ with $m \ge 3$, and set $V_i = (-g_i)g_i$ for all $i \in [1,m]$. Then
$U \t V_1 \cdot \ldots \cdot V_m$, but $U$ divides no proper subproduct. Since $U (-U) = V_1 \cdot \ldots \cdot V_m$, it follows that $\mathsf t (G_P, U) \ge \max \{2, m\} = m = |U|$. The statement on $\mathsf t (G_P)$ is an immediate consequence.
\end{proof}

\medskip
\begin{remarks} \label{3.4}
1.  Proposition \ref{3.3} shows that the property whether $H$ is locally tame or not depends only on $G_P$. This is not true for global tameness.
We argue as follows. By \cite[Example 4.13]{Ge-Ka10a}, there is a tame Krull monoid $H'$ with class group $G'$, set of prime divisors $G_P' \subset G'$ such that $D(G_P') = \infty$. Since $H'$ is tame, $\mathcal B (G_P')$ is tame by Proposition \ref{3.3}.3. By a Realization Theorem for Krull monoids (\cite[Theorem 2.5.4]{Ge-HK06a}), there is a Krull monoid $H$ with class group $G$, set of prime divisors $G_P \subset G$, and an isomorphism $\Phi \colon G \to G'$ with $\Phi (G_P) = G_P'$ such that every class in $G_P$ contains at least two distinct prime divisors.
Then \cite[Theorem 4.2]{Ge-Ka10a} implies that $H$ is not tame, but $\mathcal B (G_P) \cong \mathcal B (G_P')$ is tame.

\smallskip
2. Statement 3 of Proposition \ref{3.3} shows that the finiteness of the Davenport constant implies that $H$ is globally tame and hence locally tame. Note, if $G_P$ is
finite, then $\mathsf D (G_P)$ is finite, and  the converse holds if $G$ has finite total rank (\cite[Theorem 3.4.2]{Ge-HK06a}).
Moreover, if $G_P = G$, then $G$ is finite if and only if $\mathsf D (G)$ is finite if and only if $H$ is locally tame if and only if $H$ is tame (\cite[Theorem 4.4]{Ge-Ha08a}).
\end{remarks}

\smallskip
As mentioned above, we can have that $\mathsf t (G_P) < \infty = \mathsf t (H)$. We will also give an example of a finite abelian class group such that $\mathsf t (G) < \mathsf t (H)$ (see Theorem \ref{5.1}). Thus in general the tame degree $\mathsf t (H)$ does not coincide with the tame degree of the associated monoid of zero-sum sequences. However, if every class contains sufficiently many prime divisors, then the following proposition offers a characterization of $\mathsf t (H)$ in terms of zero-sum theory. This opens the door to study the arithmetical invariant $\mathsf t (H)$ with methods from  Combinatorial and Additive  Number Theory.

\medskip
\begin{proposition} \label{3.5}
Let $H$ be a Krull monoid with class group $G$  and let $G_P \subset G$ denote the set of classes containing prime divisors. Suppose that $G_P = - G_P$ and that $2 < \mathsf D (G_P) < \infty$.
\begin{enumerate}
\item Let $u \in \mathcal A (H)$  and  $U = \boldsymbol \beta (u)$ with $|U| \ge 3$. If every nontrivial class contains at least  $|U|+1$ distinct prime divisors, then
      \[
      \begin{aligned}
      \mathsf t (H, u) = \max \{|U|, 1+ & \min \mathsf L (A_1 \cdot \ldots \cdot A_m)  \mid  m \in \N, U = S_1 \cdot \ldots \cdot S_m \ \text{and, for all} \ i \in [1,m] , \\
                        &  \ S_i, A_i \in \mathcal F (G_P) \setminus \{1\} \ \text{with} \ S_i A_i \in \mathcal A (G_P) \} \,.
      \end{aligned}
      \]

\smallskip
\item If every nontrivial class contains at least $\mathsf D (G_P)+1$ distinct prime divisors, then
      \[
      \begin{aligned}
      \mathsf t (H) = \max \{ \mathsf D (G_P) , 1+ & \min \mathsf L (A_1 \cdot \ldots \cdot A_m)  \mid  m \in \N, A_1, \ldots, A_m \in \mathcal F (G_P) \setminus \{1\} \\
       & \text{are zero-sum free such that } \ \sigma (A_1) \cdot \ldots \cdot \sigma (A_m) \in \mathcal A (G_P) \} \,.
      \end{aligned}
      \]
      \end{enumerate}
\end{proposition}

\begin{proof}
We may suppose that $H$ is reduced, and we consider a divisor theory $H \hookrightarrow F = \mathcal F (P)$.

1. First note that
\[
\sigma (A_1 \cdot \ldots \cdot A_m) = \sigma (A_1)+\ldots+\sigma (A_m)=-\sigma (U) =0 \,,
\]
hence $A_1\cdot \ldots \cdot A_m \in \mathcal B (G_P)$. Since $S_iA_i \in \mathcal A (G_P)$ and $S_i \ne 1$, it follows that $|A_i| \le \mathsf D (G_P)-1$ for all $i \in [1,m]$,
and hence $|A_1 \cdot \ldots \cdot A_m| \le m (\mathsf D (G_P)-1) \le |U| (\mathsf D (G_P)-1)$. Thus we get
\[
\min \mathsf L (A_1 \cdot \ldots \cdot A_m) \le |A_1\cdot \ldots\cdot A_m|/ 2 \le  (\mathsf D (G_P)-1)\mathsf D (G_P)/2 \,.
\]
Thus the set $\{|U|, 1+\min \mathsf L (\cdot) \mid \ldots \}$ is finite, and we denote by $t'$ its maximum.

First we show that $\mathsf t (H,u) \le t'$. Let $a \in H$ with $\mathsf t (H,u) = \mathsf t (a,u)$. Let $\ell, m \in \N$, $v_1, \ldots, v_m, u_2, \ldots, u_{\ell} \in \mathcal A (H)$ such that $u \t v_1 \cdot \ldots \cdot v_m$, but $u$ divides no proper subproduct, $v_1 \cdot \ldots \cdot v_m = uu_2 \cdot \ldots \cdot u_{\ell}$ and $\max \{\ell, m \} = \mathsf t (a, u)$.
If $\ell \le m$, then $\mathsf t (a,u) = m \le |U| \le t'$. Suppose that $\ell > m$. Then $\mathsf t (a,u) = \ell$ and $\ell-1=\min \mathsf L (u^{-1}a)$. Since $u$ divides $v_1 \cdot \ldots \cdot v_m$ but no proper subproduct, there are $s_1, \ldots, s_m, a_1, \ldots, a_m \in F \setminus \{1\}$ such that $u = s_1\cdot \ldots \cdot s_m$, and $v_i=s_ia_i$ for all $i \in [1,m]$.
Setting $S_i = \boldsymbol \beta (s_i)$ and $A_i = \boldsymbol \beta (a_i)$ for all $i \in [1,m]$ we obtain that
\[
\begin{aligned}
\mathsf t (H,u) & = \ell = 1 + \min  \mathsf L (u^{-1}a) = 1 + \min \mathsf L \big( \boldsymbol \beta (u^{-1}a)  \big)  = 1 + \min \mathsf L (A_1 \cdot \ldots \cdot A_m) \le t' \,.
\end{aligned}
\]

Next we show that $t' \le \mathsf t (H,u)$. If $t' = |U|$, then the statement follows from Proposition \ref{3.3}. Suppose that $t' > |U|$, and let $S_1, A_1, \ldots, S_m, A_m$ be as in the definition of $t'$ such that $t' = 1 + \min \mathsf L (A_1 \cdot \ldots \cdot A_m)$. There are  $s_1, \ldots, s_m \in F \setminus \{1\}$ such that $u = s_1 \cdot \ldots \cdot s_m$ and $S_i = \boldsymbol \beta (s_i)$ for all $i \in [1,m]$. Set $\gamma = |A_1 \cdot \ldots \cdot A_m|$. Since
every class contains at least $|U|+1$ distinct prime divisors, there are primes $p_1, \ldots, p_{\gamma} \in P \setminus \supp_P(u)$ and elements $a_1, \ldots, a_m \in F$ such that
$a_1 \cdot \ldots \cdot a_m = p_1 \cdot \ldots \cdot p_{\gamma}$, $\gcd_F (u, a_1 \cdot \ldots \cdot a_m) = 1$, and $\boldsymbol \beta (a_i) = A_i$ for all $i \in [1,m]$.
Now we define $v_i = s_i a_i$ for all $i \in [1,m]$, and observe that $v_1, \ldots , v_m \in \mathcal A (H)$. By construction, $u \t v_1 \cdot \ldots \cdot v_m$, but $u$ does not divide any proper subproduct. Let $u_2, \ldots, u_{\ell} \in \mathcal A (H)$ such that $\max \{\ell, m\} \le \mathsf t (a,u)$. Then
\[
\ell \ge 1 + \min \mathsf L (u^{-1}a) = 1 + \min \mathsf L (A_1 \cdot \ldots \cdot A_m) = t' > |U| \ge m \,,
\]
and hence
\[
\mathsf t (H,u) \ge \mathsf t (a,u) = \max \{\ell, m\} = \ell \ge t' \,.
\]

\smallskip
2. Let $t'$ denote the maximum on the right hand side. First we show that for all $u \in \mathcal A (H)$, we have $\mathsf t (H, u) \le t'$. We choose $u \in \mathcal A (H)$ and use the equation for $\mathsf t (H, u)$ derived in 1. Clearly, we have $|u| \le \mathsf D (G_P) \le t'$. Let $S_1, A_1, \ldots, S_m, A_m $ be as in 1. Since $U = S_1 \cdot \ldots \cdot S_m \in \mathcal A (G_P)$, it follows that $U' = \sigma (S_1) \cdot \ldots \cdot \sigma (S_m) \in \mathcal A (G_P)$ and hence $-U' = \sigma (A_1) \cdot \ldots \cdot \sigma (A_m) \in \mathcal A (G_P)$. Thus all the assumptions on $A_1, \ldots , A_m$ of 2. are satisfied, and thus $\mathsf t (H,u) \le t'$.

Conversely, we show that $t' \le \mathsf t (H)$. If $t'=\mathsf D (G_P)$, then $t' = \mathsf D (G_P) \le \mathsf t (H,u)$ for some $u \in \mathcal A (H)$ with $|u|= \mathsf D (G_P)$. Suppose that $t' = 1 + \min \mathsf L (A_1 \cdot \ldots \cdot A_m)$ with $A_1, \ldots , A_m$ be as in 2. For $i \in [1,m]$, we define $S_i = - \sigma (A_i)$, and we set $U = S_1 \cdot \ldots \cdot S_m$.
Then $U, S_1A_1, \ldots, S_mA_m \in \mathcal A (G_P)$, and for any $u \in \boldsymbol \beta^{-1} (U)$, we have $t' = 1 + \min \mathsf L (A_1 \cdot \ldots \cdot A_m) \le \mathsf t (H,u)$ by 1.
\end{proof}

\section{Krull monoids with small and with large global tame degree}

Let $H$ be a Krull monoid with class group $G$ such that every class contains a prime divisor, and suppose that $|G| \ge 3$. Then, by Proposition \ref{3.3}, we have
\[
\mathsf D^* (G) \le \mathsf D (G) \le \mathsf t (G) \le \mathsf t (H) \le 1 + \frac{\mathsf D (G) (\mathsf D (G)-1)}{2} \,.
\]
The main result in this section is Theorem \ref{4.11}.
It characterizes when the equality  $\mathsf D^* (G) = \mathsf t (G)$ and when the equality  $\mathsf D^* (G) = \mathsf t (H)$ do hold. These characterizations reveal the first example showing that  $\mathsf t (G) < \mathsf t (H)$ may happen. On the other hand, if we consider $\mathsf t (H)$ and $\mathsf D (G)$ as functions of the rank $r$ of $G$ (with fixed exponent), then $\mathsf t (H)$ is growing as the upper bound given above.

We start with two lemmas providing lower bounds for the global tame degree.

\medskip
\begin{lemma} \label{4.1}
Let $G$ be a finite abelian group with $|G| > 1$.
\begin{enumerate}
\item Then $\mathsf t (G) \ge 1 + \exp (G) \mathsf k^* (G)$.

\smallskip
\item If $G$ is cyclic of order $|G| = n \ge 25$, then $
      \mathsf t (G) > 2n - 7 \sqrt{n}+ 10$.
\end{enumerate}
\end{lemma}

\begin{proof}
1. See \cite[Proposition 6.5.2]{Ge-HK06a}.

\smallskip
2. Let $G$ be cyclic of order $|G| = n \ge 25$. We start with a special construction (which is very similar to \cite[Proposition 6.5.2]{Ge-HK06a}).
We set $n = qm + j$, where  $q  \in [2,n-2]$ with $\gcd (q, n) = 1$, and  $m, j \in \N$, and choose a non-zero element $g \in G$. The atoms
\[
U = (qg)^n, \ U_0 = g^n, \ U_1 = (-g)g, \ V = (qg)g^{n-q}, \ V' = (qg)(-g)^q
\]
are pairwise distinct, and we have
\[
A = V^{n-m}{V'}^m = U U_1^{qm}U_0^{n-q-m} \in \mathcal B (G) \,.
\]
Since $|\mathsf Z (U_1^{qm}U_0^{n-q-m})| = 1$, $A$ has precisely one factorization which is divisible by $U$. Therefore we obtain that
\begin{equation}\label{eq_nqm}
\begin{aligned}
\mathsf t (G) & \ge \mathsf t (A, U) \ge \mathsf d ( V^{n-m}{V'}^m, U U_1^{qm}U_0^{n-q-m}) \\
& = \max \{n, 1+qm+n-q-m\} = n + (q-1)(m-1) \,.
\end{aligned}
\end{equation}
\medskip

Thus, it would remain to find $q$, $m$ and $j$ fulfilling the relevant conditions such that $n + (q-1)(m-1)$ is greater than $2n - 7 \sqrt{n} +10$. The main obstacle here is that $q$ needs to be co-prime to $n$. To side-step this problem for the most part, we first apply the first part of this result.

It is easy to see that if $n$ is divisible by at least three distinct primes then $\mathsf{k}(G) \ge 2 -1/n$, and the result follows directly from the first part.

So, we may assume that $n$ is divisible by at most two distinct primes.
For such an $n$ it is well-known and not hard to see that among each four consecutive integers there is at least one co-prime to $n$. Indeed,  consider two distinct primes $p,q$, and assume for a contradiction $a,a+1,a+2,a+3$ are all divisible by $p$ or $q$, say $p \mid a$, then $p\nmid (a+1)$ so $q\mid a+1$ implying that $p$ needs to divide $a+2$ and thus $p=2$ and $q\neq 2$, and $a+3$ is divisible neither by $p$ nor by $q$, a contradiction.

Thus, we can choose some $q$ that is co-prime to $n$ from the set
$\{ \lfloor \sqrt{n}\rfloor - 3, \dots , \lfloor \sqrt{n}\rfloor \}$; note that here we use the condition $n \ge 25$ to ensure that these elements are at least $2$.

We then set $m =\lfloor n/q \rfloor$, the quotient of the Euclidean division of $n$ by $q$, and $j$ the rest (note that $j$ is non-zero as $q$ does not divide $n$). It follows that $m \ge \lfloor \sqrt{n}\rfloor $.

From this we get that $(m-1)(q-1) > (\sqrt{n} -2) (\sqrt{n} - 5)$, and the claim follows.
\end{proof}

It is apparent from the proof that also for $n< 25$ non-trivial bounds can be obtained using the same method. However, if one wishes to have a bound for some specific (small) values of $n$, one should in any case rather use the bound $n + (q-1)(m-1)$ directly for an in this case suitable choice of $q$, or at least not use the crude estimate $\lfloor \sqrt{n} \rfloor >  \sqrt{n} - 1$ so that we made no effort to avoid the condition $n \ge 25$.
Moreover, for $n$ a prime or also a prime-power one can get somewhat better bounds in essentially the same way,
using $q= \lfloor \sqrt{n}\rfloor$ or $q \in \{\lfloor \sqrt{n}\rfloor -1, \lfloor \sqrt{n}\rfloor\}$, respectively.

\begin{remark} \
\begin{enumerate}
  \item If $G$ is cyclic of order $|G| = n \ge 5$ and $n$ is a prime number, then
      \(
      \mathsf t (G) \ge 2n - 4 \sqrt{n}+ 4.
      \)
  \item If $G$ is cyclic of order $|G| = n \ge 9$ and $n$ is a prime-power, then
      \(
      \mathsf t (G) \ge 2n - 5 \sqrt{n}+ 6.
      \)
\end{enumerate}

\end{remark}

\medskip
\begin{lemma} \label{4.2}
Let $G$ be a finite abelian group, and let $r \in \N$ be even.
\begin{enumerate}
\item Let $(e_1, \ldots, e_r) \in G^r$ be independent such that  $\gcd \big(\ord (e_i),   \ord (e_j) \big) > 1$ for all $i, j \in [1,r]$, $e_0 = -e_1 - \ldots - e_r$, $\overline e = \sum_{i=1}^r (-1)^{i+1} e_i$,
      \[
U = \ \overline e \prod_{i \in [1,r] \ \text{odd}} (-e_0-e_i) \prod_{i \in [1,r] \ \text{even}} (e_0 + e_i) \,,
      \] and
\[
\frac{r}{2} = \ord (e_i) k_i + j_i \quad \text{with} \quad j_i \in [0, \ord (e_i)-1] \quad  \text{and} \quad k_i = \Big\lfloor \frac{r}{2\ord (e_i)} \Big\rfloor \quad \text{for all} \quad i \in [1,r] \,.
\]
 Then, $U$ is a minimal zero-sum sequence and
\[
 \mathsf t (G, U) \ge 1 + \sum_{i=1}^r   \left( 2 \Big\lfloor \frac{r}{2\ord (e_i)} \Big\rfloor + j_i \right) \,.
\]

\smallskip
\item If $G = C_n^r$ with $n \ge 2$ and $\gcd (r - 1, n)=1$, then $\mathsf t (G) \ge 1 + 2r \lfloor \frac{r}{2n} \rfloor$ and $\mathsf t (C_2^r) \ge 1 + \frac{r^2}{2}$.
\end{enumerate}
\end{lemma}

\begin{proof}
1. We set $S = \overline{e}^{-1}U$. Since $\sigma (S) = -\overline{e}$, it follows that $U = \overline{e}S$ is a zero-sum sequence. To show that $U$ is a minimal zero-sum sequence it remains to verify that $S$ is zero-sumfree.  Assume to the contrary that $S$ has a non-empty zero-sum subsequence
\[
T = \prod_{i \in I_o} (-e_0-e_i) \prod_{i \in I_e} (e_0 + e_i) \,, \quad \text{where $I_o$ and $I_e$ are subsets of $[1,r]$ of odd and even numbers, resp.}
\]
For $i \in [1,r]$ let $a_i \in [0,\ord(e_i)-1]$ such that $\sigma(T) = \sum_{i=1}^r a_i e_i$.
Recalling the definition of $e_0$ we infer that
\[
a_je_j = (|I_o| - |I_e| + \delta_j ) e_j  \quad \text{where} \quad \delta_j = \begin{cases}  -1 &  \text{ for } j \in I_o, \\
  1 & \text{ for } j \in I_e, \\
 0  & \text{ otherwise.}
\end{cases}
\]
Since $\sigma(T)= 0$ and $(e_1, \ldots, e_r)$ is independent, it follows that $a_je_j= 0 $ for each $j\in [1,r]$, that is  $\ord(e_j) \mid (|I_o| - |I_e| + \delta_j )$.

Now, since $T$ is non-empty not all $\delta_j$ equal $0$. However, this implies that no $\delta_j$ equals $0$. Indeed, if  $\delta_{k}=0$ for some $k$, then  $\ord(e_k) \mid |I_o| - |I_e|$, and considering some $k'$ such that $\delta_{k'} \neq 0$ we infer that $\ord (e_{k'})$ divides $|I_o| - |I_e|+1$ or $|I_o| - |I_e|-1$. This entails that $\gcd \big( \ord(e_k),\ord(e_{k'}) \big)$ divides two consecutive integers, a contradiction.  Consequently, $I_o \cup I_e = [1,r]$,  $T=S$, and thus $\sigma (T) = \sigma(S)=\overline{e} \neq 0$, a contradiction.

We set $V_0 = \overline e \prod_{\nu=1}^r (-1)^{\nu} e_{\nu}$, and we define
\[
V_i = \begin{cases}
(-e_0-e_i) \prod_{\nu \in [1,r] \setminus \{i\}} (-e_{\nu})\,, & \quad \text{if $i$ is odd,} \\
      (e_0+e_i) \prod_{\nu \in [1,r] \setminus \{i\}} e_{\nu}\,, & \quad \text{if $i$ is even.}
      \end{cases}
\]
By construction, we have $U \t \prod_{\nu=0}^r V_{\nu}$,  but $U$ does not divide any proper subproduct. Furthermore, we have
\[
W = U^{-1} \prod_{\nu=0}^r V_{\nu} = \prod_{\nu=1}^r \big( (-e_{\nu}) e_{\nu} \big)^{r/2} \,.
\]
Next we study $\mathsf L (W)$. For any nonzero $g \in G$, $k \in \mathbb N_0$, and $j \in [0, \ord (g)-1]$, we have
\[
\mathsf L \Big(  \big( (-g)g \big)^{k \ord (g)+j} \Big) = \{ 2k + \nu(\ord (g)-2) + j \mid \nu \in [0,k] \} \,,
\]
and the minimum of this set equals $2k+j$. Thus, for every $i \in [1,r]$, we obtain that
\[
\mathsf L \Big(  \big( (-e_i)e_i \big)^{r/2} \Big) = \big\{ 2k_i + \nu(\ord (e_i)-2) + j_i \mid \nu \in [0,k_i] \big\} \,,
\]
with
\[
\min \mathsf L \Big(  \big( (-e_i)e_i \big)^{r/2} \Big) = 2    \Big\lfloor \frac{r}{2\ord (e_i)} \Big\rfloor + j_i \,.
\]
Since
\[
\mathsf L (W) = \sum_{i=1}^r \mathsf L \Big(  \big( (-e_i)e_i \big)^{r/2} \Big) \,,
\]
$\min \mathsf L (W)$ is the sum of the minima, and it follows that
\[
\mathsf t (G, U) \ge 1 + \min \mathsf L(W) =  1 + \sum_{i=1}^r \Big( 2  \Big\lfloor \frac{r}{2\ord (e_i)} \Big\rfloor + j_i  \Big) \,.
\]

\smallskip
2. Suppose that $G = C_n^r$ with $r$ and $n$ as above. After choosing a basis $(e_1, \ldots, e_r)$ of $G$ with $\ord (e_1) = \ldots = \ord (e_r) = n$, the first inequality follows immediately from 1. Now suppose that $n=2$.  If $r \equiv 0 \mod 4$, then the statement on $\mathsf t (C_2^r)$ follows from the first statement. If $r = 4k + 2$ with $k \in \mathbb N_0$, then $r/2 = 2k + 1$, and 1. implies that
\[
\mathsf t (G) \ge 1 + r \big( 2k+1 \big) = 1 + \frac{r^2}{2} \,. \qedhere
\]
\end{proof}

\medskip
\begin{lemma} \label{4.3}
Let $G$ be a finite abelian group. Then $\exp (G) \mathsf k^* (G) \ge \mathsf d^* (G)$,
and equality holds if and only if $G$ is a $p$-group of the form $G = C_n^r$ where $n, r \in \N$.
\end{lemma}

\begin{proof}
By definition, the statement holds if $|G| = 1$. Suppose that $|G|> 1$, say
\[
G \cong C_{n_1} \oplus \dots \oplus C_{n_{r}} \cong C_{q_1} \oplus \dots \oplus C_{q_{s}}\,,
\]
where $r, s \in \mathbb N$, $n_1, \ldots,n_{r} \in \N$, \ $1 < n_1 \t \ldots
\t n_{r}$, and $q_1, \ldots , q_{s}$ are
prime powers.
Note that $\exp (G) = n_r = \lcm (q_1, \ldots, q_s)$. Obviously, the statement holds for cyclic groups of prime power order.
Suppose that $G$ is cyclic but not a $p$-group. Then $r=1$ and $s \ge 2$. Since $\frac{q_j-1}{q_j} \ge \frac{1}{2}$ for all $j \in [1,s]$, it follows that
\[
\sum_{j=1}^s \frac{q_j-1}{q_j} \ge 1 \quad \text{and hence} \quad n_r \sum_{j=1}^s \frac{q_j-1}{q_j} \ge n_r > \sum_{i=1}^r (n_i-1) \,.
\]
Thus the statement holds for cyclic groups. If $G$ is not cyclic, then
\[
\exp (G) \mathsf k^* (G) = \exp (G) \sum_{i=1}^r \mathsf k^* (C_{n_i}) \overset{(1)}{\ge} \sum_{i=1}^r n_i \mathsf k^* (C_{n_i}) \overset{(2)}{\ge} \sum_{i=1}^r \mathsf d^* (C_{n_i}) = \mathsf d^* (G) \,,
\]
where equality in $(1)$ holds if and only if $n_1=\ldots=n_r$ and equality in $(2)$ holds if and only if $n_1, \ldots, n_r$ are prime powers.
\end{proof}

\medskip
\begin{proposition} \label{4.4}
Let $G_1, G_2$ be  finite abelian groups with  $\mathsf t (G_1) > \mathsf D (G_1)$.
\begin{enumerate}
\item If $\mathsf t (G_2) > \mathsf D (G_2)$ or $\mathsf D (G_2) = \mathsf D^* (G_2)$, then $\mathsf t (G_1 \oplus G_2) \ge \mathsf t (G_1) + \mathsf t (G_2) - 1$.

\smallskip
\item If  $\mathsf d (G_1 \oplus G_2) = \mathsf d^* (G_1) + \mathsf d^* (G_2)$, then $\mathsf t (G_1 \oplus G_2) > \mathsf D (G_1 \oplus G_2)$.
\end{enumerate}
\end{proposition}

\begin{proof}
1. Let $i \in [1,2]$.
By definition of $\mathsf t (G_i)$, there exists an $U^{(i)} \in \mathcal A (G_i)$ with $\mathsf t (G_i) = \mathsf t (G_i, U^{(i)})$. By definition of $\mathsf t (G_i, U^{(i)})$, there are $\ell_i, m_i \in \mathbb N$, $U_2^{(i)}, \ldots, U_{\ell_i}^{(i)}, V_1^{(i)}, \ldots V_{m_i}^{(i)} \in \mathcal A (G_i)$ such that $U^{(i)} \t V_1^{(i)} \cdot \ldots \cdot V_{m_i}^{(i)}$, but $U^{(i)}$ divides no proper subproduct, $U^{(i)} U_2^{(i)} \cdot \ldots \cdot U_{\ell_i}^{(i)} = V_1^{(i)} \cdot \ldots \cdot V_{m_i}^{(i)}$, and $\mathsf t (G_i, U^{(i)}) = \max \{\ell_i, m_i \}$.
If  $\mathsf t (G_i, U^{(i)}) = \mathsf t (G_i) > \mathsf D (G_i)$, then   $m_i \le |U^{(i)}| \le \mathsf D (G_i)$ implies that $m_i < \ell_i = \mathsf t (G_i, U^{(i)})$ and
\[
\ell_i - 1 = \min \mathsf L ( U_2^{(i)} \cdot \ldots \cdot U_{\ell_i}^{(i)} ) \,.
\]

Now suppose that $\mathsf t (G_2) = \mathsf D (G_2) = \mathsf D^* (G_2)$. Then we provide a new construction of the above type where we have $\ell_2 \ge m_2$ and $\ell_2 - 1 = \min \mathsf L ( U_2^{(2)} \cdot \ldots \cdot U_{\ell_2}^{(2)} ) $. For simplicity we use the same notation as above.
Let $(e_1, \ldots, e_r)$ be a basis of $G$ such that $\ord (e_i) = n_i$ for all $i \in [1,r]$, $1 < n_1 \t \ldots \t n_r$ and $\mathsf d^* (G) = \sum_{i=1}^r (n_i-1)$. We set $e_0 = e_1+ \ldots + e_r$,
\[
U^{(2)} = (-e_0)e_0 , \ V_1^{(2)} = e_0 \prod_{i=1}^r e_i^{n_i-1} \ , \quad \text{and} \quad V_2^{(2)} = - V_1^{(2)} \,.
\]
Then we get $U^{(2)} \prod_{i=1}^r \big( (-e_i)e_i \big)^{n_i-1} = V_1^{(2)} V_2^{(2)}$, $\mathsf t (G_2, U^{(2)}) = \mathsf t (G_2) = 1 + \mathsf d^* (G)$, $\ell_2 = \mathsf d^* (G)+1 \ge 2 = m_2$ and
\[
\ell_2 - 1 = \min \mathsf L ( U_2^{(2)} \cdot \ldots \cdot U_{\ell_2}^{(2)} ) \,,
\]
with the obvious definition of $U_2^{(2)}, \ldots,  U_{\ell_2}^{(2)}$.

We continue simultaneously for both cases.
For $\nu \in [1, m_i]$, we set $V_{\nu}^{(i)} = S_{\nu}^{(i)}A_{\nu}^{(i)}$, with $S_{\nu}^{(i)}, A_{\nu}^{(i)} \in \mathcal F (G_i)$ such that $U^{(i)} = S_1^{(i)} \cdot \ldots \cdot S_{m_i}^{(i)}$.
We choose an element $g_i \in G_i$ with $g_i \t S_1^{(i)}$, and define
\[
U = g_1^{-1}g_2^{-1}(g_1+g_2)U^{(1)} U^{(2)} \quad \text{and} \quad V_1 = g_1^{-1}g_2^{-1}(g_1+g_2)V_1^{(1)} V_1^{(2)} \,.
\]
Then $U, V_1 \in \mathcal A (G_1 \oplus G_2)$ and $U \t V_1 V_2^{(1)} \cdot \ldots \cdot V_{m_1}^{(1)} V_2^{(2)} \cdot \ldots \cdot V_{m_2}^{(2)}$, but $U$ divides no proper subproduct. For $W = U^{-1}V_1 V_2^{(1)} \cdot \ldots \cdot V_{m_1}^{(1)} V_2^{(2)} \cdot \ldots \cdot V_{m_2}^{(2)}$ we get
\[
W = U_2^{(1)} \cdot \ldots \cdot U_{\ell_1}^{(1)}U_2^{(2)} \cdot \ldots \cdot U_{\ell_2}^{(2)}
\]
and
\[
\mathsf L (W) = \mathsf L (U_2^{(1)} \cdot \ldots \cdot U_{\ell_1}^{(1)} ) + \mathsf L (U_2^{(2)} \cdot \ldots \cdot U_{\ell_2}^{(2)} ) \,.
\]
This show that $\min \mathsf L (W) = (\ell_1-1) + ( \ell_2-1)$. Summing up we obtain that
\[
\mathsf t (G_1 \oplus G_2) \ge \mathsf t (G_1 \oplus G_2, U) \ge \max \{m_1+m_2, \ell_1+\ell_2-1\} = \ell_1+\ell_2-1= \mathsf t (G_1) + \mathsf t (G_2) - 1 \,.
\]

\smallskip
2. Using 1. we infer that
\[
\begin{aligned}
\mathsf t (G_1 \oplus G_2) & \ge \mathsf t (G_1) + \mathsf t (G_2) - 1 \ge  \mathsf D (G_1) + \mathsf D (G_2) \\
 & \ge \mathsf d^* (G_1) + \mathsf d^* (G_2) + 2  = \mathsf d (G_1  \oplus G_2) + 2 = \mathsf D (G_1 \oplus G_2) + 1 \,.
\end{aligned} \qedhere
\]
\end{proof}

\medskip
{\it For the rest of this section,
let $H$ be a reduced  Krull monoid, $H \hookrightarrow \mathcal F (P)$ a divisor theory with class group $G$ and suppose that every class contains a prime divisor.}
\medskip

\medskip
\begin{lemma} \label{4.5}
Let $G = C_2^r$ with $r \ge 3$ and $\ell \in [1, r+1]$. Let $A_1, \ldots, A_{\ell}$ be
pairwise distinct zero-sum free sequences with $|A_i| = r$. Then
there exist some $k \in [0, \ell-1]$ and $U_1, \ldots , U_k \in
\mathcal A (G)$ with $U_1 \cdot \ldots \cdot U_k \mid A_1 \cdot
\ldots \cdot A_{\ell}$ such that
\[
|U_1 \cdot \ldots \cdot U_k| \ge 3(\ell-1) .
\]
\end{lemma}

\begin{proof}
See \cite[Lemma 6.6.5]{Ge-HK06a}.
\end{proof}

\medskip
\begin{lemma} \label{4.6}
Let $G = C_2^r$ with $r \ge 3$, $n \ge 3$, and $u, v_1, \ldots , v_n \in \mathcal A (H)$ such
that $u \mid v_1 \cdot \ldots \cdot v_n$ and does not divide any proper subproduct. Furthermore, for every $i
\in [1,n]$, suppose that $v_i = s_i a_i$ with $a_i, s_i \in \mathcal F (P) \setminus \{1\}$
such that $u =  \prod_{i=1}^n s_i$ and set $w = u^{-1} v_1 \cdot
\ldots \cdot v_n$. Then $\min \mathsf L (w) \le \frac{n(r-1)+1}{2}$.
\end{lemma}

\begin{proof}
We set $U = \boldsymbol \beta (u)$, $W = \boldsymbol \beta (w)$, $V_i = \boldsymbol \beta (v_i)$, $S_i = \widetilde{\boldsymbol \beta} (s_i)$ and $A_i = \widetilde{\boldsymbol \beta} (a_i)$ for all $i \in [1,n]$. It is sufficient to verify the upper bound for $\min \mathsf L (W)$.
After renumbering if necessary there is some $\ell  \in [0,n]$ such that $|A_1| = \ldots = |A_{\ell}| = r$ and $|A_i| \le r-1$ for all $i
\in [\ell+1, n]$. If $\ell = 0$, then $\min \mathsf L (W) \le |W|/2 = n(r-1)/2$. Let $\ell \ge 1$, and assume to the contrary that there are distinct $i, j \in [1, \ell]$ such that $A_i = A_j$, say $i=1$ and $j=2$. Then
\[
\sigma (S_1) = \sigma (A_1) = \sigma (A_2) = \sigma (S_2) \,,
\]
hence $u = s_1s_2$ and $n \le 2$, a contradiction. Thus $A_1, \ldots, A_{\ell}$ are pairwise distinct zero-sum free sequences.
By  Lemma \ref{4.5} there exist some $k \in [0,
\ell-1]$ and $U_1, \ldots , U_k \in \mathcal A (G)$ such that
\[
U_1 \cdot \ldots \cdot U_k \mid A_1 \cdot \ldots \cdot A_{\ell}
\quad \text{and} \quad
|U_1 \cdot \ldots \cdot U_k| \ge 3(\ell-1) .
\]
Setting $W' = (U_1 \cdot \ldots \cdot U_k)^{-1}W$ we infer that
\[
|W'| \le |W| - 3(\ell-1) \le \ell r + (n-\ell)(r-1) -3(\ell-1) =
n(r-1) - 2(\ell-1) +1 .
\]
Thus $W$ has a factorization of length at most
\[
k + \frac{|W'|}{2} \le \frac12 \Big( n(r-1) + 2(k-(\ell-1))+1 \Big)
\le \frac{n(r-1)+1}{2} \,. \qedhere
\]
\end{proof}

\medskip
\begin{lemma} \label{4.7}
Let $G = C_2^r$ with $r \ge 3$, and suppose that $\mathsf t (G) \ge 2 + \frac{r(r-1)}{2}$. Then there are $U, V_1, \ldots , V_{r+1} \in \mathcal A (G)$, where $U \t V_1 \cdot \ldots \cdot V_{r+1}$ but $U$ divides no proper subproduct, such that the   following properties are satisfied{\rm \,:}
\begin{enumerate}
\item[(a)] $U = e_1 \cdot \ldots \cdot e_{r+1}$, and $V_i = e_i A_i$ where $A_i \in \mathcal F (G)$ and $e_i = \gcd (U, V_i)$ for all $i \in [1,r+1]$.

\smallskip
\item[(b)] $A_i A_j$ is not zero-sum free for all $i, j \in [1, r+1]$ distinct.

\smallskip
\item[(c)] For $W = U^{-1}V_1 \cdot \ldots \cdot V_{r+1}$ we have $\gcd (U, W) = 1$ and $\mathsf t (G) = \mathsf t (G, U) = 1 + \min \mathsf L (W)$.
\end{enumerate}
\end{lemma}

\begin{proof}
Let $U \in \mathcal A (G)$ with $\mathsf t (G) = \mathsf t (G, U) \ge 2 + \frac{r(r-1)}{2}$. Then there are $V_1, \ldots, V_m \in \mathcal A (G)$ with $U \t V_1 \cdot \ldots \cdot V_m$, $U \nmid \prod_{i \in I}V_i$ for any $I \subsetneq [1,m]$, and such that
\[
\mathsf t (G, U) = 1 + \min \mathsf L (W) \,.
\]
Then Lemma \ref{4.6} (applied with $H = \mathcal B (G)$) implies that
\[
1 + \frac{r(r-1)}{2} \le \min \mathsf L (W) \le \frac{m(r-1)+1}{2} \,,
\]
and hence $m = r+1 = |U|$. Therefore we may assume that Property (a) holds.

Assume to the contrary that Property (b) fails. Then there exist some $i, j \in [1, r+1]$
distinct such that $A_i A_j$ is zero-sum free, say $i=1$ and $j=2$.
We set
\[
\overline U = (e_2+e_1) \prod_{i=3}^{r+1} e_i \quad \text{and}
\quad
\overline{V_2} = (e_2+e_1) A_2 A_1 .
\]
Then $\overline U, \overline{V_2} \in \mathcal A (G)$, $\overline
U \mid \overline{V_2} V_3 \cdot \ldots \cdot V_{r+1}$ and
\[
\overline{ U}^{-1} \overline{V_2}V_3 \cdot \ldots \cdot V_{r+1}
= W = U^{-1} V_1 \cdot \ldots \cdot V_{r+1} .
\]
Thus Lemma \ref{4.6} implies that $\min \mathsf L (W) \le
\frac{r(r-1)}{2}$, a contradiction.
Finally we assume to the contrary that Property (c) fails.  This means that the set $I \subset [1, r+1]$, defined as
\[
\prod_{i \in I} e_i = \gcd (U,W) \,,
\]
is nonempty. Let $i \in I$. Then $e_i$
divides $W = A_1 \cdot \ldots \cdot A_{r+1}$. If $e_i \mid
A_i^{-1} W$, then $U \mid \prod_{j \in [1, r+1] \setminus \{i\}}
V_j$, a contradiction. Thus $e_i \mid A_i$ whence $V_i = e_i^2$.
If there would exist $i, j \in I$  distinct, then   $A_i A_j = e_i e_j$ would be
zero-sum free. This implies that $|I|
\le 1$.

Since we assumed $I$ to be nonempty, we get that  $|I| = 1$, say $I = \{r+1\}$. Then for $i \in [1, r]$
we have
\[
V_i = e_i A_i \quad \text{and} \quad V_{r+1} = e_{r+1}^2 \,.
\]
After renumbering if necessary,  we may suppose that for some $\ell \in [0,r]$ we
have $|A_1| = \ldots = |A_{\ell}| = r$ and $|A_i| \le r-1$ for all $i
\in [\ell+1, r]$.
 If $\ell = 0$, then $\min \mathsf L (W) \le |W|/2 \le (1 + r(r-1))/2$, a contradiction. Suppose that $\ell \ge 1$.
By  Lemma  \ref{4.5}, there exists some $k \in [1,
\ell-1]$ and $U_1, \ldots , U_k \in \mathcal A (G)$ such that
\[
U_1 \cdot \ldots \cdot U_k \mid A_1 \cdot \ldots \cdot A_{\ell} \quad
\text{and} \quad |U_1 \cdot \ldots \cdot U_k| \ge 3(\ell-1) .
\]
Setting $W' = (U_1 \cdot \ldots \cdot U_k)^{-1} W$ we infer that
\[
|W'| \le |W| - 3(\ell-1) \le \ell  r + (r-\ell)(r-1) + 1 - 3(\ell-1) = r(r-1)
-2\ell + 4 .
\]
Let $W' = U_{k+1} \cdot W''$ with $U_{k+1} \in \mathcal A (G)$ and
$e_{r+1} \mid U_{k+1}$. Since $U \nmid \prod_{i=1}^r V_i$,
$e_{r+1}$ occurs exactly once in $W'$ which implies that
$|U_{k+1}| \ge 3$ and $|W''| = |W'| - |U_{k+1}| \le r(r-1) -2 \ell +
1$. Thus $W$ has a factorization with length at most
\[
k+1 + \frac{|W''|}{2} = \frac{1}{2} \left( r(r-1) + 2(k+1-\ell) + 1
\right) \le \frac{r(r-1)}{2} + \frac{1}{2}
\]
whence $\min \mathsf L (W) \le \frac{r(r-1)}{2}$, a contradiction.
\end{proof}

\medskip
\begin{lemma} \label{4.8}
Let $G = C_2^r$ with $r \in \N$.
\begin{enumerate}
\item $\mathsf t (G) = \mathsf D (G)$ if and only if $r \in [2, 3]$.

\smallskip
\item If $r=2$, then $\mathsf t (H)=\mathsf D (G)$.

\smallskip
\item If $r=3$ and if there is a nontrivial class containing at least two distinct prime divisors, then $\mathsf t (H) = \mathsf D (G)+1$.
\end{enumerate}
\end{lemma}

\begin{proof}
We proceed in four steps, distinguishing the cases $r=1$, $r=2$, $r=3$, and $r \ge 4$.

(i) If $r=1$, then $\mathcal B (G)$ is factorial and hence $\mathsf t (G) = 0 < \mathsf D (G) = 2$.

\smallskip
(ii) Suppose that  $r = 2$. Then $\mathsf D (G) = 3$, and and by Proposition \ref{3.3} we have
\[
3 \le \mathsf t (H) \le  \max \left\{ \mathsf t (G), \Big\lfloor \frac{3+(\mathsf D (G_P)-1)^2}{2} \Big\rfloor \right\} = \max \{ \mathsf t (G), 3 \} \,,
\]
and hence it suffices to verify that $\mathsf t (G) \le 3$.
This can be done by  a quick direct check.

\smallskip
(iii) Suppose that $r=3$. To show the statement on $\mathsf t (G)$, we assume to the contrary that $\mathsf t (G) \ge \mathsf D (G)+1 = r+2=5 = 2 + \frac{3 \cdot 2}{2}$. Then let $U, V_1, \ldots, V_5$ have all the properties of Lemma \ref{4.7}, and use all the notations of that lemma. In particular, we have $U = e_1 \cdot \ldots \cdot e_4$. Then $(e_1, e_2, e_3)$ is a basis of $G$, $e_4=e_1+e_2+e_3$, and
$
G = \{0, e_1, e_2, e_3, e_1+e_2+e_3, e_2+e_3, e_1+e_3, e_1+e_2 \}$.
Since $\gcd (U,W)=1$, it follows that
\[
\supp (A_1) \subset \supp (W) \subset \{e_2+e_3, e_1+e_3, e_1+e_2 \} \,.
\]
On the other hand, $A_1$ is zero-sum free with $\sigma (A_1) = e_1$ and with $|A_1| \in [2,3]$, a contradiction.

Now suppose that there is a nontrivial class containing at least two distinct prime divisors. First we show that $\mathsf t (H) \ge 5$. Let $(e_1, e_2, e_3)$ be a basis of $G$ and let $p_i \in P \cap e_i$ for all $i \in [1,3]$, and let $p_3' \in P \cap e_3$ with $p_3' \ne p_3$. Let $u = q_1q_2q_3p_3 \in \mathcal A (H)$ with   $\boldsymbol \beta (u) = (e_1+e_2+e_3)(e_1+e_3)(e_2+e_3)e_3$ such that $q_1 \in P \cap (e_1+e_2+e_3), q_2 \in P \cap (e_1+e_3)$, and $q_3 \in P \cap (e_2+e_3)$. Now we define
\[
v_1 = q_1p_1p_2p_3', \  v_2 = q_2 p_1p_3',\  v_3 = q_3p_2p_3', \quad  \text{and} \quad v_4 = p_3p_3 \,.
\]
Then $v_1, v_2, v_3, v_4 \in \mathcal A (H)$, $u \t v_1v_2v_3v_4$, but $u$ does not divide any proper subproduct. Since we have $\mathsf L (u^{-1}v_1v_2v_3v_4) = \{4\}$, it follows that
\[
\mathsf t (H,u) \ge \mathsf t (a,u) \ge \max \{4, 1 + \min \mathsf L (u^{-1}a)\} = 5 = \mathsf D (G)+1 \,.
\]

Assume to the contrary that $\mathsf t (H) > 5$. We choose $a \in H$ and $u \in \mathcal A (H)$ such that $\mathsf t (H) = \mathsf t (a, u) \ge 6$. Then there are $u_2, \ldots, u_{\ell}, v_1, \ldots, v_m \in \mathcal A (H)$ such that $u \t v_1 \cdot \ldots \cdot v_m$, but $u$ divides no proper subproduct, and $\max \{\ell, m \} = \mathsf t (a, u)$. We set $U = \boldsymbol \beta (u)$, $w = u^{-1}a$, and $W = \boldsymbol \beta (W)$. Since $m \le |U| \le \mathsf D (G) = 4$, it follows that $\mathsf t (a,u) = \ell = 1 + \min \mathsf L (W) \ge 6$.
From this we get that $|U| = 4$ and $|W| \in [10,12]$. Then for every $i \in [1,4]$, there are $p_i \in P$ and $a_i \in F \setminus \{1\}$ such that $v_i = p_ia_i$ and $u = p_1p_2p_3p_4$.
We set $A_i = \boldsymbol \beta (a_i)$ for all $i \in [1,4]$, and after renumbering if necessary there is an $s \in [0,4]$ such that $|A_1| = \ldots = |A_s| = 3$, and $3 > |A_{s+1}| \ge \ldots \ge |A_4|$. Note that $W = A_1 \cdot \ldots \cdot A_4$, and since $\sigma (A_1), \ldots, \sigma (A_4)$ are pairwise distinct, the sequences $A_1, \ldots, A_4$ are pairwise distinct.
Since $|A_1 \cdot \ldots \cdot A_4|=|W| \in [10,12]$, $|W| = 10$ implies $s \ge 2$, $|W|=11$ implies $s \ge 3$, and $|W|=12$ implies $s=4$.
By Lemma \ref{4.5} there exist $k \in [0, s-1]$ and $W_1, \ldots, W_k \in \mathcal A (G)$ such that $W_1 \cdot \ldots \cdot W_k \t A_1 \cdot \ldots \cdot A_s$ and $|W_1 \cdot \ldots \cdot W_k| \ge 3(s-1)$. This implies that
\[
5 \le \min \mathsf L (W) \le k + \frac{|W|-|W_1\cdot \ldots \cdot W_k|}{2} \le (s-1) + \frac{|W| - 3(s-1)}{2} = \frac{|W|-(s-1)}{2} \,,
\]
a contradiction.

\smallskip
(iv)  Suppose that $r \ge 4$. If $r \ge 4$ is even, then Lemma \ref{4.2}.2 shows that
\[
\mathsf t (G) \ge 1 + \frac{r^2}{2} > r+1 = \mathsf D (G) \,.
\]
If $r \ge 5$ is odd, then again by Lemma \ref{4.2}.2 we get that
\[
\mathsf t (G) \ge \mathsf t (C_2^{r-1}) \ge 1 + \frac{(r-1)^2}{2} > r+1 = \mathsf D (G) \,. \qedhere
\]
\end{proof}

\medskip
\begin{lemma} \label{4.9}
Let $G = C_3^r$ with $r \in \N$. Then $\mathsf t (H) = \mathsf D (G)$ if and only if $r = 1$.
\end{lemma}

\begin{proof}
Let $r=1$. Then $\mathsf D (G)=3$, and by Proposition \ref{3.3} we have
\[
3 \le \mathsf t (H) \le  \max \left\{ \mathsf t (G), \Big\lfloor \frac{3+(\mathsf D (G_P)-1)^2}{2} \Big\rfloor \right\} = \max \{ \mathsf t (G), 3 \} \,,
\]
and hence it suffices to check that $\mathsf t (G) \le 3$. Let
 $U \in \mathcal A (G)$. If $|U| = 2$, then $\mathsf t (G, U) \le 1 + \frac{|U|(\mathsf D (G)-1)}{2} = 3$. If $|U| = 3$, then $U = g^3$, $V_1=V_2=V_3 = (-g)g$ for some nonzero $g \in G$, and hence $\mathsf t (G, U) = 2$.

Note that $\mathsf D (C_3^r) = \mathsf D^* (C_3^r) = 2r+1$. Since $\mathsf t (H) \ge \mathsf t (G)$,  Proposition \ref{4.4} implies that  it is sufficient to show that $\mathsf t (C_3 \oplus C_3) > 5 = \mathsf D (C_3 \oplus C_3)$.

Let $G = C_3 \oplus C_3$, and let $(e_1, e_2)$ be a basis of $G$. We define
\[
V_1 = V_2 = e_1 (-e_1-e_2)^2(e_2-e_1)^2, \ V_3 = V_4 = e_2 (-e_1-e_2)^2 (e_1-e_2)^2 , \ V_5 = (e_1+e_2)(-e_1)^2 (e_1-e_2) \,,
\]
and
\[
U = e_1^2 e_2^2 (e_1+e_2) \,.
\]
Then $U, V_1, \ldots, V_5 \in \mathcal A (G)$, $U \t V_1 \cdot \ldots \cdot V_5$, but $U$ does not divide any proper subproduct. We set $W = U^{-1}V_1 \cdot \ldots \cdot V_5$, and assert that $\min \mathsf L (W) = 6$, which implies that
\[
\mathsf t (G) \ge \mathsf t (G, U) \ge 1 + \min \mathsf L (W) = 7 \,.
\]

Note $W = (-e_1-e_2)^8(e_1-e_2)^5(e_2-e_1)^4(-e_1)^2$, in particular is has lenght $19$.
We determine the atoms $S \in \mathcal A (G)$ with $S \t W$ and $|S| \ge 4$.
Such an atom must not contain both $(e_1-e_2)$ and $(e_2-e_1)$, yet contains at least three distinct elements;
consequently it contains $-e_1$. First, suppose the two elements besides $-e_1$ are $(-e_1-e_2)$ and $(e_2-e_1)$.
We note that $(-e_1)(-e_1-e_2)(e_2-e_1)$ is a (minimal) zero-sum sequence, and thus the only minimal zero-sum sequence with this support.
Thus, since $|S|\ge 4$, we have $\supp(S) = \{-e_1, -e_1-e_2, e_1-e_2\}$.
If the multiplicity of $-e_1$ is $1$, we get the atom $(-e_1)(-e_1-e_2)(e_1-e_2)^2$, and
if the this multiplicity is $2$, we get the atom $(-e_1)^2(-e_1-e_2)^2(e_1-e_2)$.

Therefore, noting that the multiplicty of $-e_1$ in $W$ is $2$, we can infer that every
factoriazation of $W$ contains (counted with multiplicity) either one atom of lengths $5$ and none of lengths $4$ or none of length $5$
and at most $2$ of length $4$.
Thus, $\min \mathsf{L}(W) $ is at least the smaller of $ 1 + \lceil(19-5)/3 \rceil = 6$ and $2 + \lceil(19-2\cdot 4)/3 = 6$;
that is it is at least $6$.
\end{proof}

\medskip
\begin{lemma} \label{4.10}
Let $G = C_4^r$ with $r \in \N$. Then $\mathsf t (H) = \mathsf D (G)$ if and only if $r=1$.
\end{lemma}

\begin{proof}
Let $r=1$. Let $u \in \mathcal A (H)$ but not prime. We have to show that $\mathsf t (H, u) \le \mathsf D (G) = 4$. Let $v_1, \ldots, v_m \in \mathcal A (H)$ such that $u \t v_1 \cdot \ldots \cdot v_m$, but $u$ divides no proper subproduct. We set $w = u^{-1}v_1 \cdot \ldots \cdot v_m \in H$, and note that $m \le |u|$ and $m \in [2,4]$. If $m = 2$, then $|u^{-1}v_1v_2| \le 6$, and hence $\min \mathsf L (w) \le 3$. We set $U = \boldsymbol \beta (u)$, $W = \boldsymbol \beta (w)$, $V_i = \boldsymbol \beta (v_i)$ for all $i \in [1,m]$, and distinguish the cases $m=3$ and $m=4$.

\noindent CASE 1: \,$m=3$.

First suppose that $|U|=3$. Then $U = g^2(2g)$ for some $g \in G$ with $\ord (g) = 4$, and we may suppose that $g \t V_1$, $g \t V_2$, and $(2g) \t V_3$. Then $|V_3| \le 3$. Assume to the contrary that $\min \mathsf L (W) \ge 4$. Then $|W| \ge 8$, which implies that $|V_1| = |V_2| = 4$ and $|V_3| = 3$. Then $V_1=V_2 = g^4$, and $V_3 \in \{ (2g)g^2, (2g)(-g)^2 \}$. In both cases we get that $\min \mathsf L (W) < 4$, a contradiction.

Now suppose that $|U|=4$. Then $U = g^4$  for some $g \in G$ with $\ord (g) = 4$, and we assume again that $\min \mathsf L (W) \ge 4$. This implies that $|W| = 8$, and hence, after renumbering if necessary,  $V_1 \in \{g^4,  g^2(2g)\}$, and $V_2=V_3 = g^4$.  Thus we get $\mathsf L (W) = \{2\}$, a contradiction.

\noindent CASE 2: \,$m=4$.

Then $U = g^4$  for some $g \in G$ with $\ord (g) = 4$. Thus $V_1, V_2, V_3, V_4  \in \{ g(-g), g^2 (2g), g^4 \}$. Assume to the contrary that $\min \mathsf L (W) \ge 4$. Then $|W| \ge 8$, and at most two of the $V_i$ are equal to $(-g)g$. Discussing all possibilities we quickly see that $\min \mathsf L (W) \le 3$, a contradiction.

\smallskip
Now suppose that $r \ge 2$, and
note that $\mathsf D (C_4^r) = \mathsf D^* (C_4^r) = 3r+1$. Since $\mathsf t (H) \ge \mathsf t (G)$, it suffices to prove that $\mathsf t (G) > \mathsf D (G)$. Thus by Proposition \ref{4.4} it is sufficient to show that $\mathsf t (C_4 \oplus C_4) > 7 = \mathsf D (C_4 \oplus C_4)$.

Let $G = C_4 \oplus C_4$ and let $(e_1, e_2)$ be a basis of $G$. We define
\[
\begin{aligned}
V_1  = V_2 = V_3 & = e_1 (-e_1-e_2)^3 (2e_1-e_2) \,, \\
V_4  = V_5 = V_6 & = e_2 (-e_1-e_2)^3 (-e_1+2e_2) \,,   \\
V_7 & = (e_1+e_2)^4 \,, \qquad \text{and} \qquad
U   = e_1^3 e_2^3 (e_1+e_2) \,.
\end{aligned}
\]
Then $U, V_1, \ldots, V_7 \in \mathcal A (G)$, $U \t V_1 \cdot \ldots \cdot V_7$, but $U$ does not divide any proper subproduct. We set $W = U^{-1}V_1 \cdot \ldots \cdot V_7$, and assert that $\min \mathsf L (W) = 7$, which implies that
\[
\mathsf t (G) \ge \mathsf t (G, U) \ge 1 + \min \mathsf L (W) = 8 \,.
\]
First we determine the atoms $S \in \mathcal A (G)$ with $S \t W$. Since
\[
(2e_1-e_2, -e_1+2e_2) = (e_1, e_2) \left( \begin{array}{cc}
2 & -1 \\
-1 &  2 \end{array} \right)
\]
and the determinant of the transformation matrix equals $-1$ modulo $4$, it follows that $(2e_1-e_2, -e_1+2e_2)$ is independent, and hence
the sequence $(2e_1-e_2)^3 (-e_1+2e_2)^3$ is zero-sum free.
Now it is easy to check that
\[
S_1 = (e_1+e_2)(2e_1-e_2)^3 (-e_1+2e_2)^3 , S_2 = (e_1+e_2)^2 (2e_1-e_2)^2 (-e_1+2e_2)^2 , S_3 = (e_1+e_2)^3(2e_1-e_2) (-e_1+2e_2)
\]
are the atoms $S$ with $(-e_1-e_2) \nmid S$ and $S \t W$, and
\[
S_4 = (-e_1-e_2)(2e_1-e_2) (-e_1+2e_2), S_5 = (-e_1-e_2)(e_1+e_2) , S_6 = (-e_1-e_2)^4
\]
are the atoms $S$ with $(-e_1-e_2) \t S \t W$. We claim that
\[
\mathsf Z (W) = \{ S_1 S_5^2 S_6^4, S_3 S_4^2 S_6^4, S_2 S_4S_5 S_6^4, S_4^3 S_5^3 S_6^3 \}
\]
which implies that $\mathsf L (W) = \{ 7, 9\}$. Clearly, it remains to show that the given factorizations are the only ones. Let $z \in \mathsf Z (W)$. If $S_1 \t z$, then obviously $z = S_1 S_5^2 S_6^4$. Suppose that $S_1 \nmid z$. If $S_3 \t z$, then $z = S_3 S_4^2 S_6^4$. Suppose that $S_3 \nmid z$. If $S_2 \t z$, then $z = S_2 S_4S_5 S_6^4$. If also $S_2 \nmid z$, then $z = S_4^3 S_5^3 S_6^3$.
\end{proof}

\medskip
\begin{theorem} \label{4.11}
Let $H$ be a Krull monoid with finite class group $G$ such that every class contains a prime divisor.
\begin{enumerate}
\item  $\mathsf t (G) = \mathsf D^* (G)$ if and only if $G \in \{ C_3, C_4, C_2^2, C_2^3 \}$.

\smallskip
\item If one nontrivial class contains at least two distinct prime divisors, then $\mathsf t (H) = \mathsf D^* (G)$ if and only if $G \in \{C_2, C_3, C_4, C_2^2\}$.

\smallskip
\item Suppose that $G$ has rank $\mathsf r (G) = r$, and consider  both, $\mathsf D (G)$ and $\mathsf t (H)$, as  functions  in $r$. Then there are constants $M_1, M_2, M_3, M_4 \in \mathbb R_{>0}$ (depending only on $\exp (G)$ but
      not on $r$) such that
      \[
      M_1 r \le \mathsf D (G) \le M_2r \quad \text{and} \quad M_3 r^2 \le \mathsf t (H) \le  M_4 r^2 \,.
      \]
      In particular, $\mathsf t (H)$  grows as the upper bound $1 + \mathsf D (G) ( \mathsf D (G)-1) /2$, given in Proposition \ref{3.3}.3.
\end{enumerate}
\end{theorem}

\begin{proof}
If $|G|=1$, then both, $H$ and $\mathcal B (G)$, are factorial, whence $\mathsf t (H)=\mathsf t (G)=0$, but we have $\mathsf D^* (G)=1$.
From now on we suppose that $|G|>1$, say $G \cong C_{n_1} \oplus \dots \oplus C_{n_{r}}$,
where $r = \mathsf r (G) \in \mathbb N$ is the rank of $G$, $n_1, \ldots,n_{r} \in \N$, \ $1 < n_1 \t \ldots
\t n_{r}$, and  $n= n_r = \exp (G)$ is the exponent of $G$.

1. If $|G|=2$, then $\mathsf t (G)=0$ and $\mathsf D^* (G) = 2$.
Suppose that $|G|>2$, and that $\mathsf t (G) = \mathsf D^* (G)$. By Lemma \ref{4.1} we obtain that
$\mathsf t (G) \ge 1 + \exp (G) \mathsf k^* (G)$.
Therefore Lemma \ref{4.3} implies that equality holds and that $G$ is a $p$-group with
$n_1= \ldots = n_r$, whence $G = C_n^r$. If $n \ge 5$, then $\mathsf t (C_n) > n = \mathsf D (C_n)$ by Lemma \ref{4.1}.2. Thus  Proposition \ref{4.4}.2 implies that $\mathsf t (C_n^r) > \mathsf D (C_n^r)$. Therefore it remains to consider the cases where $n \in [2,4]$.  Lemmas \ref{4.8}, \ref{4.9}, and \ref{4.10} show that the mentioned groups satisfy $\mathsf t (G) = \mathsf D^* (G)$, and that there are no other groups $G$ with $\exp (G) \le 4$ with this property.

\smallskip
2. If $|G|=2$, then $H$ is not factorial whence $\mathsf t (H) \ge 2$, and Proposition \ref{3.3}.3 implies that $\mathsf t (H) \le 2$. If $G = C_2^3$, then Lemma \ref{4.8} implies that
$\mathsf t (H) = \mathsf D (G)+1$. Since $\mathsf t (H) \ge \mathsf t (G) \ge \mathsf D(G) \ge \mathsf D^* (G)$, the remaining assertions follow from 1.

\smallskip
3. We have $
1+ r (n_1-1) \le \mathsf D^* (G) \le \mathsf D (G) \le \mathsf D (C_n^r)$, and by \cite[Theorem 5.5.5]{Ge-HK06a},
we obtain that
\[
\mathsf D (C_n^r) \le n +  n \log n^{r-1} \le (n \log n) r \,.
\]
Thus there exist constants $M_1, M_2$ as required. Let $p$ be a prime with $p \t n_1$. Then $G$ has a subgroup isomorphic to $C_p^r$ and hence $\mathsf t (C_p^r) \le \mathsf t (G)$. We intend to find a $M_3 \in \R_{>0}$ with $M_3 r^2 \le \mathsf t (C_p^r)$. If $p=2$, this holds by Lemma \ref{4.2}.2. Let $p$ be odd. Then there is an $s \in [r-2,r]$ such that $s$ is even and $\gcd (p,s-1) = 1$, and hence
Lemma \ref{4.2}.2 implies that
\[
1+ 2 s \Big\lfloor \frac{s}{2p} \Big\rfloor \le \mathsf t (C_{p}^s) \le \mathsf t (C_p^r) \,.
\]
Finally Proposition \ref{3.3}.3 implies that
\[
\mathsf t (H) \le 1 + \frac{\mathsf D (G) \big(\mathsf D (G)-1 \big)}{2} \le \mathsf D (G)^2 \le M_2^2 r^2 \,. \qedhere
\]
\end{proof}

\section{Krull monoids whose class group is either cyclic or an elementary $2$-group} \label{5}

In this section we study Krull monoids $H$ whose class group $G$ is either cyclic or an elementary $2$-group. We get quite precise results, which confirm the general tendency of the tame degree indicated by Theorem \ref{4.11}. Suppose  $|G| \ge 3$, that every class contains a prime divisor, and consider again the inequality
\[
\mathsf D (G) \le \mathsf t (H) \le 1 + \frac{\mathsf D (G) \big( \mathsf D (G)-1 \big)}{2} \,.
\]
In case of elementary $2$-groups, $\mathsf t (H)$ almost equals the upper bound and, apart from one exceptional  case, we always have $\mathsf t (H) = \mathsf t (G)$. Suppose that $G$ is cyclic of order $|G|=n \ge 5$. Then $\mathsf D (G) = n \le \mathsf t (H) \le n^2$ (for better lower bounds see Lemma \ref{4.1}). As expected, it turns out that the tame degree is close to the lower bound.

\medskip
\begin{theorem} \label{5.1}
Let $H$ be a Krull monoid whose class group $G$ is an elementary
$2$-group, say $G \cong C_2^r$ with $r \in \mathbb N$, and
suppose that every class contains a prime divisor. Then we have
\begin{enumerate}
\item If $r=1$, then $\mathsf t (H) = 2$ and $\mathsf t (G) = 0$.

\smallskip
\item If $r=3$, then $\mathsf t (G) = 4$, and if  one nontrivial class contains at least two distinct prime divisors, then $\mathsf t (H) = 5$.

\item
\[
\mathsf t (H) = \mathsf t (G) \quad \begin{cases} = 1 + \frac{r^2}{2} & \quad \text{if} \ \ r \ge 2 \quad \text{is even,}  \\
                                                   \ge 2 + \frac{r(r-1)}{2} & \quad \text{if} \ \  r \ge 5 \quad
\text{is odd.}
\end{cases}
\]
\end{enumerate}
\end{theorem}

\begin{proof}
We may suppose that $H$ is reduced, and that $H \hookrightarrow \mathcal F (P)$ a divisor theory with
class group $G$.
All statements of 1. and 2. follow from Lemma \ref{4.8} and from Theorem \ref{4.11}, and hence it remains to prove 3.
The assertion for even $r \ge 2$ follows from \cite[Corollary 6.5.6]{Ge-HK06a}.
Suppose that  $r \ge 5$ is odd. The
lower bound for  $\mathsf t (G)$ follows from \cite[Theorem
6.5.3]{Ge-HK06a}.
So it remains to show that $\mathsf t (H) =
\mathsf t (G)$.  By Proposition \ref{3.3}, it suffices to
show that $\mathsf t (H) \le \mathsf t (G)$.

Let $u \in \mathcal A (H)$. We have to show that
$\mathsf t (H, u) \le \mathsf t (G)$. If $u \in P$, then $\mathsf t
(H, u) = 0$. Suppose that $u \notin P$, and let $a \in uH$, $v_1, \ldots , v_n \in \mathcal A (H)$ with $z = v_1 \cdot \ldots \cdot v_n \in \mathsf Z (a)$ such that $u$ divides
no proper subproduct of $v_1 \cdot \ldots \cdot v_n$, and such that $\mathsf t (H, u) = \max \{n, 1 + \min \mathsf L (w)\}$, where $w = u^{-1}v_1\cdot \ldots \cdot v_n$.
If $n \le 2$, then the assertion follows. Suppose that $n \ge 3$, and note that $n \le |u|$.
Then Lemma \ref{4.6} implies that $\min \mathsf L (w) \le \frac{n(r-1)+1}{2}$. Thus we are done for $n \le r$.

Suppose that $n=r+1$. Then $u = p_1 \cdot \ldots \cdot
p_{r+1}$ and, for all $i \in [1,r+1]$, $v_i =p_ia_i$ where $p_i \in P$ and $a_i \in \mathcal F (P) \setminus \{1\}$.
For $i \in [1, r+1]$, we set $V_i = \boldsymbol \beta (v_i)$, $g_i =
[p_i]$ and we set $U = \boldsymbol \beta (u)$. Then $U \t V_1 \cdot
\ldots \cdot V_{r+1}$. After renumbering if necessary we may assume
that there is some $m \in [1, r+1]$ such that $U \t V_1 \cdot \ldots
\cdot V_m$ but $U$ does not divide any proper subproduct.

If $m=1$, then $V_1=U$, $\overline{\boldsymbol \beta}(z) = V_1 \cdot \ldots \cdot V_{r+1} \in U \mathsf Z (G)$, and by definition of the tame degree in the fibres, there is a $z' \in \mathsf Z(a) \cap u \mathsf Z(H)$ with $\overline{\boldsymbol \beta} (z') = \overline{\boldsymbol \beta} (z)$ and $\mathsf d (z,z') \le \mathsf t (H, u, \boldsymbol \beta)$. By Lemma \ref{3.2}, we get $\mathsf t (H, u, \boldsymbol \beta) \le 1 + \mathsf D (G) = r+2 \le \mathsf t (G)$.

Suppose that $m \ge 2$.
There
exist $U_2, \ldots, U_k \in \mathcal A (G)$ such that $V_1 \cdot
\ldots \cdot V_m = U U_2 \cdot \ldots \cdot U_k$ and
\[
\max \{k,m\} = \mathsf d (V_1 \cdot \ldots \cdot V_m, U U_2 \cdot
\ldots \cdot U_k) \le \mathsf t (G,U) \le \mathsf t (G) \,.
\]
Since $\boldsymbol \beta (u^{-1}a) = U_2 \cdot \ldots \cdot U_k
V_{m+1} \cdot \ldots \cdot V_{r+1}$ and $\boldsymbol \beta \colon H
\to \mathcal B(G)$ is a transfer homomorphism, there exist $u_2,
\dots , u_k, \, w_{m+1}, \dots , w_{r+1} \in \mathcal A (H)$ such
that $\boldsymbol \beta (u_i) = U_i$ for all $i \in [2,k]$, \
$\boldsymbol \beta (w_j)= V_j$ for all $j \in [m+1, r+1]$ and
$u^{-1}a = u_2 \cdot \ldots \cdot u_k w_{m+1} \cdot \ldots \cdot
w_{r+1}$. Then
\[
z' = uu_2 \cdot \ldots \cdot u_k w_{m+1} \cdot \ldots \cdot w_{r+1}
\in \mathsf Z(a) \cap u\mathsf Z(H) \quad \text{and} \quad \mathsf
d(z,z') \le \max \{r+1, k+r+1-m\}\,.
\]
If $m=r+1$, this implies \ $\mathsf d(z,z') \le \mathsf t (G)$.
If $m=2$, then
\[
k-1 \le \Big\lfloor \frac{|V_1V_2|-|U|}{2} \Big\rfloor \le \Big\lfloor \frac{2\mathsf D (G)-3}{2} \Big\rfloor = r-1 \,,
\]
and
\[
k+1+r-m \le 2r-1 \le 2 + \frac{r(r-1)}{2}  \le \mathsf t (G) \,.
\]
Suppose that $m \in [3, r]$. Then Lemma \ref{4.6} (applied with $H = \mathcal B (G)$ and $n=m \ge 3$) implies that $k-1 \le \lfloor \frac{m(r-1)+1}{2} \rfloor$. If $m \le r-1$, then
\[
k+r+1-m \le 2 + \frac{m(r-1)+1}{2} + r- m \le 2 + \frac{r(r-1)}{2}  \le \mathsf t (G) \,,
\]
If $m=r$, then $k-1 \le \frac{r(r-1)}{2}$ and
\[
k+r+1-m = k+1 \le 2 + \frac{r(r-1)}{2}  \le \mathsf t (G) \,.
\]
Thus in both cases we get $\mathsf d (z, z') \le \mathsf t (G)$.
\end{proof}

\medskip
For  $m,n \in \mathbb N$,  let $\omega (n)$ denote the number of distinct prime divisors of $n$, and let $\phi_m(n)$ denote the number
of integers $a\in [1,m]$ with $\gcd (a,n)=1$; this function is sometimes called Legendre's totient function. Thus $\phi_n(n)=\phi(n)$ is Euler's totient function.

\medskip
\begin{theorem} \label{5.2}
Let $H$ be a Krull monoid having a cyclic class group $G$ of order $|G|=n \ge 5$.
\begin{enumerate}
\item If $n = p \in \mathbb P$, then $\mathsf t (H) \le  1
+ \frac{2(p-1)p}{p+5}+2(p-1)(\frac{1}{2}+\log (\frac{p+3}{2}))$.

\smallskip
\item If $n = p^{\alpha}$, where $p \in \mathbb P$ and  $\alpha \ge 2$, then
      \[
      \mathsf t (H) \le 1
-2\alpha+\frac{2p^{\alpha+1}}{p-1}+2\alpha
n+3\sum_{i=1}^{\alpha}(p^i-1)\log\frac{p^i}{2}\,.
      \]

\smallskip
\item If $n = p_1^{\alpha_1} \cdot \ldots \cdot p_r^{\alpha_r}$, where $r \ge 2$, $\alpha_1, \ldots, \alpha_r \in \N$, and  $p_1, \ldots, p_r \in \mathbb P$ are distinct, then $\mathsf t (H) \le$
    \[
      1 + \frac{4.3}{2}\sum_{1<d|n}(d-1)+n\sum_{1<d|n,d\le
4375}\frac{d}{2}+n\sum_{4376\le
d|n}2^{\omega(d)+1}\sqrt{2\omega(d)}+3.3\sum_{1<d|n}(d-1) \log
(\lfloor\frac{d+1}{2^{\omega(d)+1}\sqrt{2\omega(d)-1}+1}-1\rfloor) \,.
    \]
\end{enumerate}
\end{theorem}

\medskip
We need a new combinatorial invariant and a series of lemmas. The proof of Theorem \ref{5.2} will be given at the end of this section.

\medskip
\begin{definition} \label{5.3}
Let $G$ be a finite abelian group. For every $t\in [2, \mathsf
D(G)]$, let $\mathsf m(G,t)$ denote the smallest integer $ \ell \in \N$ such that every sequence $S$ over $G \setminus \{0\}$ of length $|S| \ge \ell$ and in addition satisfying
$\mathsf{v}_g(S) \le \ord (g)$ for each $g \in $ $G \setminus \{0\}$,
has a minimal zero-sum subsequence $T$ of length  $|T| \ge t$.
\end{definition}

The idea behind defining this constant is to somehow quantify how easy or hard it is for a given group $G$ to avoid the existence of long minimal zero-sum subsequences. While it is clear that some additional condition, beyond the usual one on the length, is needed to make this definition a meaningful one, regarding the precise condition there is some flexibility. The one we choose is, except for excluding $0$, the most permissive one that seems reasonable. If one cares about minimal zero-sum sequences one never has a need for an element more than its order times. To exclude $0$ makes sense for the present application and more generally is convenient; the variant of the constant where $0$ would be admitted (with multiplicity $1$) would merely differ by exactly $1$ from the current version.

\medskip

The following lemma establishes some basic properties of this new invariant.

\begin{lemma} \label{5.4}
Let $G$ be a finite abelian group with $|G|>1$.
\begin{enumerate}
\item For every $t\in [2, \mathsf D(G)]$, we have  $\mathsf m(G,t) \ge \mathsf D (G)$.

\smallskip
\item We have $\mathsf m(G,2)= \mathsf D(G)$ and $\mathsf m(G,3) \in [2 \mathsf D^* (G)-1, 2 \mathsf D(G)-1]$.
\end{enumerate}
\end{lemma}

\begin{proof}
1. By definition of $\mathsf D (G)$, there is a zero-sum free sequence $S$ over $G$ of length $|S| = \mathsf D (G)-1$; note that, $S$ being zero-sum free, $\mathsf{v}_g(S) \le \ord(g)$ for each $g$. Since such a sequence does not satisfy the defining property of $\mathsf m (G, t)$, it follows that $\mathsf m(G,t) \ge \mathsf D (G)$ for every $t\in [2, \mathsf D(G)]$.

\smallskip
2. Every sequence $S$ over $G\setminus \{0\}$ of length $\mathsf D (G)$ has a zero-sum subsequence, and hence a minimal zero-sum subsequence $T$. Since $0 \nmid S$, we get $|T| \ge 2$. Thus $\mathsf m (G,2) \le \mathsf D (G)$, and equality follows by 1.

In order to show that $\mathsf m (G,3) \le 2 \mathsf D (G)-1$, let $S$ be a sequence of length $|S| \ge 2 \mathsf D (G)-1$. We write $S$ in the form $S = S_1S_2T_1T_2$ where, for $i \in [1,2]$, $S_i$ is a sequence over $G$ such that $\supp(S_i)\cap \supp(-S_i)=\emptyset$ and $T_i$ is a squarefree sequence over $G$ containing only elements of order $2$. (Recall that the multiplicity of an element of order $2$ in $S$ is at most $2$, and also note that $S_1S_2$ cannot contain elements of order $2$ appearing in $S$.)

Without restriction we may suppose that $|S_1| \ge |S_2|$ and $|T_1| \ge |T_2|$.
Then $|S_1T_1| \ge \mathsf{D}(G)$, and $S_1T_1$ thus contains a minimal zero-sum subsequence $T$. By construction we have $|T| \ge 3$.

Finally, we verify that $\mathsf m (G, 3) \ge 2 \mathsf D^* (G)-1$. Suppose that $G = C_{n_1} \oplus \ldots \oplus C_{n_r}$ with
$1 < n_1 \t \ldots \t n_r$, and let $(e_1, \ldots, e_r)$ be a basis of $G$ with $\ord (e_i) = n_i$ for all $i \in [1, r]$. Then
\[
S = \prod_{i=1}^r e_i^{n_i-1}
\]
is a zero-sum free sequence of length $|S| = \mathsf D^* (G)-1$, and the sequence $(-S)S$ fulfills the additional condition (for slightly different reasons in the cases $\ord(e_i)=2$ and $\ord(e_i)>2$), and has no minimal zero-sum subsequence $T$ of length $|T| \ge 3$. Thus it follows that $\mathsf m (G,3) > |(-S)S| = 2 \mathsf D^* (G)-2$.
\end{proof}

\medskip
{\it From now on till the rest of this section, let $G$ be a finite cyclic group of order $|G|=n \ge 5$, $G^{\bullet} = G\setminus \{0\}$, and $m \in [1,n]$.}
\medskip

\begin{lemma} \label{5.5}
Let $S$ be a
sequence over $G$ such that $\ord (g) = n$ for all $g \in \supp (S)$. If
$|S|\ge \frac{\phi(n)(n-1)+1}{\phi_m(n)}$, then $S$ has a
minimal zero-sum subsequence of length at least $\lceil\frac
nm\rceil$.
\end{lemma}

\begin{proof}
For $e\in G$  with $\ord (e) =n$, we have $S= {m_1^e}e \cdot \ldots \cdot {m_k^e}e$ where $m_1^e, \ldots, m_k^e \in [1,n]$. For every $\ell \in [2,n-1]$,
let $S(\ell,e)$ denote the subsequence consisting of all terms $m_k^e e$
with $m_k^e \in [1,\ell]$. Clearly, if
\begin{equation} \label{eq6.3.1}
|S(\ell,e)|\ge n
\end{equation}
then $S(\ell,e)$ has a minimal zero-sum subsequence of length at
least $\lceil\frac{n}{\ell} \rceil$. So, it suffices to prove the existence of some $e$ such that
(\ref{eq6.3.1}) holds  with $\ell=m$.  Since $\ord(g)=n$ for each $g\t S$ we have
\[
\sum_{e\in G , \ord(e)=n}S(m,e)=|S|\sum_{i\in [1,m],
\gcd(i,n)=1}1=\phi_{m}(n)|S| \,.
\]
Therefore,
\[
\max \{|S(m,e)| \mid e \in G, \ord(e)=n \}\ge
\frac{\phi_{m}(n)|S|}{\phi(n)} \,.
\]
It follows from $|S|\ge \frac{\phi(n)(n-1)+1}{\phi_m(n)}$ that
$\max \{| S(m,e)| \mid e\in G, \ord(e)=n \}>n-1$. This proves
(\ref{eq6.3.1}) holds for some $e$ and completes the proof.
\end{proof}

\medskip

The following technical lemma establishes some bounds on sums that is needed several times later on; the somewhat unusual indexing is convenient then.

\begin{lemma}\label{sum}
Let $c_2, \dots , c_M$ be non-negative reals.
Let $C$ be such that
\(
\sum_{i=2}^m i c_i \le C m
\)
for each $m \in [2,M]$. Then
\(
\sum_{i=2}^m  c_i \le C (1 + \sum_{i=3}^m 1/i)
\) for each $m \in [2,M]$.
\end{lemma}

\begin{proof}
The argument is by induction on $M$. For $M=2$ the claim is obvious.
Consider $M\ge 3$. Set $K_m = \sum_{i=2}^m c_i$. It suffices to show the claimed bound for $K_M$ (for the others the claim is clear by hypothesis).
Note that $M K_M = \sum_{m=2}^{M-1}K_m + \sum_{i=2}^{M}ic_i$.
Thus,
\[MK_M \le C \sum_{m=2}^{M-1} (1 + \sum_{i=3}^m 1/i) + CM = CM(1 + \sum_{i=3}^M 1/i),\]
where the last equality can be seen, for example, by another inductive argument.
\end{proof}

\begin{lemma} \label{5.6}
Let $n =p \in \mathbb P$.
\begin{enumerate}
\item For every $t \in [1, p-1]$, we have $m(G,t+1)\le \left\lfloor
\frac{(p-1)^2}{\lfloor p/t\rfloor} \right\rfloor+1$.

\smallskip
\item If $S$ is a zero-sum sequence over $G^{\bullet}$, then $\min \mathsf L (S)\le \min
\{\frac{|S|}{2},\frac{2|S|}{p+5}+2(p-1)(\frac{1}{2}+\log
(\frac{p+3}{2}))\}$.
\end{enumerate}
\end{lemma}

\begin{proof}
1. This follows from Lemma \ref{5.5}.

\smallskip
2.
Clearly, we have $\min\mathsf L (S)\le \max\mathsf L (S) \le
\frac{|S|}{2}$. Thus it suffices to prove that $\min\mathsf L (S)$ is bounded above by the second term in the above set. To do so, we show that there exists a factorization  $S = U_1 \cdot
\ldots \cdot U_t$, where  $U_1, \ldots,
U_t \in \mathcal A (G)$ and $t$ is bounded above by the given term.
We construct $U_1, \ldots, U_t$ recursively. Indeed, for $i \in [1,t]$, let $U_i$
be a minimal zero-sum subsequence of  $S(\prod_{j=1}^{i-1}U_j)^{-1}$, whose length is maximal possible. Now we use 1. to obtain an upper bound on $t$. For every $k\in [2,p]$, let
$n_k$ be the number of $U_i$ such that $|U_i|=k$. For every $m\in
[2,p-1]$, the construction of $U_i$ and 1. imply that
\begin{equation} \label{eq6.5.1}
\sum_{i=2}^min_i\le \left\lfloor \frac{(p-1)^2}{\lfloor p/m\rfloor}
\right\rfloor \,.
\end{equation}
If $m\le \frac{p+3}{2}$, then $\lfloor p/m\rfloor\ge
\frac{p-m+1}{m} \ge \frac{p-1}{2m}$. Therefore, from
(\ref{eq6.5.1}) we infer that
\begin{equation} \label{eq6.5.2}
\sum_{i=2}^min_i \le 2(p-1)m
\end{equation}
holds for every $m\in [2,\frac{p+3}{2}]$.

By equation \eqref{eq6.5.2} and Lemma \ref{sum} we obtain that
\[
\sum_{i=2}^mn_i\le 2(p-1)(1+\frac{1}{3}+\frac{1}{4}+\ldots
+\frac{1}{m})
\]
holds for every $m\in [2,\frac{p+3}{2}]$. Especially,
\[
\sum_{i=2}^{\frac{p+3}{2}}n_i\le
2(p-1)(\frac{1}{2}+\sum_{i=2}^{\frac{p+3}{2}}\frac{1}{i})\le
2(p-1)(\frac{1}{2}+\log (\frac{p+3}{2})) \,.
\]
Since  $\sum_{i\ge \frac{p+5}{2}}in_i\le |S|$, we have
$\sum_{i\ge \frac{p+5}{2}}n_i\le \frac{2|S|}{p+5}$. Hence,
\[
t = \sum_{i=2}^pn_i\le \frac{2|S|}{p+5}+2(p-1)(\frac{1}{2}+\log
(\frac{p+3}{2}))\,. \qedhere
\]
\end{proof}

\medskip
\begin{lemma}\label{5.7}
Let $n=p^{\alpha}$,  where $p \in \mathbb P$ and $\alpha \ge 2$,
and let $S$ be a zero-sum sequence over $G$
such that $\ord (g) = n$ for all $g \in \supp (S)$. Then $\min \mathsf L
(S)\le \min
\{\frac{|S|}{2},\frac{2|S|}{n+1}+3(n-1)(\frac{1}{2}+\log
(\frac{n}{2}))\}$.
\end{lemma}

\begin{proof}
As in  Lemma \ref{5.6}.2., it suffices  to
show that there exists a factorization  $S = U_1 \cdot \ldots
\cdot U_t$, where $U_1, \ldots, U_t \in \mathcal A (G)$
and $t$ is bounded above by the second term in the above set.
We construct $U_1, \ldots, U_t$ recursively. Indeed, for $i \in [1,t]$, let $U_i$
be a minimal zero-sum subsequence of $S(\prod_{j=1}^{i-1}U_j)^{-1}$, whose length is maximal possible. We are going to use Lemma
\ref{5.5}  to get an upper bound on $t$. For every $k\in [2,n]$, let
$n_k$ be the number of $U_i$ such that $|U_i|=k$. For every $m\in
[3,n-1]$, the construction of $U_i$ and Lemma \ref{5.5} imply that
\begin{equation} \label{equ6.7.1}
\sum_{i=2}^{m-1}in_i\le \frac{\phi(n)(n-1)}{\phi_{\lfloor\frac
nm\rfloor}(n)}
\end{equation}
If $m\le \frac{n+2}{2}$ then $\phi_{\lfloor\frac nm\rfloor}(n)
=\lfloor\frac nm\rfloor-\lfloor\frac{\lfloor\frac
nm\rfloor}{p}\rfloor\ge \lfloor\frac nm\rfloor(1-\frac 1p)\ge
\frac{n-m+1}{m}(1-\frac 1p)\ge \frac{n}{2m}(1-\frac
1p)=\frac{\phi(n)}{2m}$. It follows from (\ref{equ6.7.1}) that
\[
\sum_{i=2}^{m-1}in_i \le  2m(n-1)\le 3(m-1)(n-1) \,.
\]
Therefore, for every $m\in [2,\frac{n}{2}]$, we have
\begin{equation} \label{equ6.7.2}
\sum_{i=2}^min_i \le 3(n-1)m
\end{equation}
 It follows from \eqref{equ6.7.2}, applying Lemma \ref{sum}, that
\[
\sum_{i=2}^mn_i\le 3 (n-1) (1+\frac{1}{3}+\frac{1}{4}+\ldots
+\frac{1}{m})
\]
holds for every $m\in [2,\frac{n}{2}]$. Especially,
\[
\sum_{i=2}^{\frac{n}{2}}n_i\le
3(n-1)(\frac{1}{2}+\sum_{i=2}^{\frac{n}{2}}\frac{1}{i})\le
3(n-1)(\frac{1}{2}+\log (\frac{n}{2})) \,.
\]
Since  $\sum_{i> \frac{n}{2}}in_i=\sum_{i\ge \frac{n+1}{2}}in_i\le
|S|$ , $\sum_{i> \frac{n}{2}}n_i\le \frac{|S|}{\frac{n+1}{2}}$.
Hence,
\[
t = \sum_{i=2}^nn_i\le \frac{2|S|}{n+1}+3(n-1)(\frac{1}{2}+\log
(\frac{n}{2}))\,. \qedhere
\]
\end{proof}

\medskip
\begin{lemma} \label{5.8}
Let $n=p^{\alpha}$,  where $p \in \mathbb P$   and $\alpha\ge 2$,
and let $S$ be a zero-sum sequence over
$G^{\bullet}$. For every positive divisor $d>1$ of $n$, let
$N_d$ denote the number of the terms of $S$ which have order $d$.
Then $\min \mathsf L (S)\le \min \{\frac{|S|}{2},
-2\alpha+\frac{2p^{\alpha+1}}{p-1}+\sum_{i=1}^{\alpha}\frac{2N_{p^i}}{p^i+1}
+3\sum_{i=1}^{\alpha}(p^i-1)\log\frac{p^i}{2}\}$.
\end{lemma}

\begin{proof}
Clearly, it
suffices to prove that $\min\mathsf L (S)$ is bounded above by the second term in the above set.
For every $i\in [1,\alpha]$, let $S_i$ denote the subsequence of $S$
consisting of all terms with order $p^i$, let $T_i$ be a zero-sum subsequence of $S_i$ with maximal possible length, and set $T_i'=S_iT_i^{-1}$. Therefore
\[
S=  S_1 \cdot \ldots \cdot S_{\alpha} = \prod_{i=1}^{\alpha}T_i'\prod_{i=1}^{\alpha}T_i \quad \text{and} \quad
\prod_{i=1}^{\alpha}T_i' \ \text{ has sum zero}.
\]
By the maximality of $T_i$ we infer that $
|T_i'|\le p^i-1$
for every $i\in [1,\alpha]$. Hence,
\begin{equation}\label{eq6.8.1}
|\prod_{i=1}^{\alpha}T_i'|\le
\sum_{i=1}^{\alpha}(p^i-1) \le  p^{\alpha}\frac{p}{p-1}-\alpha \,.
\end{equation}
Therefore,
\[
\min \mathsf L (S)\le \min \mathsf L
(\prod_{i=1}^{\alpha}T_i')+\sum_{i=1}^{\alpha}\min \mathsf L
(T_i)\le
\frac{|\prod_{i=1}^{\alpha}T_i'|}{2}+\sum_{i=1}^{\alpha}\min \mathsf
L (T_i) \,.
\]
It follows from (\ref{eq6.8.1}) that
\begin{equation} \label{eq6.8.2}
\min \mathsf L (S)\le
\frac{p^{\alpha}\frac{p}{p-1}-\alpha}{2}+\sum_{i=1}^{\alpha}\min
\mathsf L (T_i) \,.
\end{equation}
By Lemma \ref{5.6}.2 and Lemma \ref{5.7} we obtain that
\[
\min \mathsf L (T_i)\le
\frac{2|T_i|}{p^i+1}+3(p^i-1)(\frac{1}{2}+\log \frac{p^i}{2})
\]
holds for every $i\in [1,\alpha]$. It follows from (\ref{eq6.8.1})
that
\[
\begin{array}{ll} \min \mathsf L (S) &\le
\frac{p^{\alpha}\frac{p}{p-1}-\alpha}{2}+\sum_{i=1}^{\alpha}(\frac{2|T_i|}{p^i+1}+3(p^i-1)(\frac{1}{2}+\log
\frac{p^i}{2})) \\
&=\frac{p^{\alpha+1}-\alpha(p-1)}{2(p-1)}+\sum_{i=1}^{\alpha}\frac{2|T_i|}{p^i+1}+
\frac{3}{2}\sum_{i=1}^{\alpha}(p^i-1)+3\sum_{i=1}^{\alpha}(p^i-1)\log\frac{p^i}{2}
\\ & =2\frac{p^{\alpha+1}-\alpha(p-1)}{p-1}+\sum_{i=1}^{\alpha}\frac{2|T_i|}{p^i+1}+
3\sum_{i=1}^{\alpha}(p^i-1)\log\frac{p^i}{2} \\ & \le
2\frac{p^{\alpha+1}-\alpha(p-1)}{p-1}+\sum_{i=1}^{\alpha}\frac{2N_{p^i}}{p^i+1}+
3\sum_{i=1}^{\alpha}(p^i-1)\log\frac{p^i}{2} \\ &
=-2\alpha+\frac{2p^{\alpha+1}}{p-1}+\sum_{i=1}^{\alpha}\frac{2N_{p^i}}{p^i+1}+
3\sum_{i=1}^{\alpha}(p^i-1)\log\frac{p^i}{2} \,. \qedhere
\end{array}
\]
\end{proof}

We need certain bounds for Legendre's totient function. We establish what we need in the two subsequent lemmas in a self-contained way.

\begin{lemma} \label{5.9}
Let $n=p_1^{\alpha_1}\cdot \ldots \cdot
p_s^{\alpha_s}$,  where $s\ge 2$, $\alpha_1, \ldots, \alpha_s \in \mathbb N$, and $p_1, \ldots, p_s \in \mathbb P$ are distinct. If $m\ge 2^{s+1}\sqrt{2s-1}$ then $\phi_m(n) \ge
\frac{m}{2}\prod_{i=1}^s(1-\frac{1}{p_i})=\frac{m\phi(n)}{2n}$.
\end{lemma}

\begin{proof}
By the inclusion-exclusion principle we know that
\[
\begin{array}{ll} \phi_m(n) &=m-\sum_{i=1}\lfloor\frac{m}{p_i}\rfloor+\sum_{1\le
i<j\le
s}\lfloor\frac{m}{p_ip_j}\rfloor-\ldots+(-1)^s\lfloor\frac{m}{p_1\cdot \ldots \cdot
p_s}\rfloor \\ & \ge m-\sum_{i=1}\frac{m}{p_i}+\sum_{1\le i<j\le
s}(\frac{m}{p_ip_j}-1)-\sum_{1\le
i<j<k}\frac{m}{p_ip_jp_k}+\sum_{1<i<j<k<l}(\frac{m}{p_ip_jp_kp_l}-1)-\ldots
\\ &=m\prod_{i=1}^s(1-\frac{1}{p_i})-  ({s \choose 2} +{s \choose
4}+\ldots)=m\prod_{i=1}^s(1-\frac{1}{p_i})-2^{s-1} \,.
\end{array}
\]
Therefore,
\begin{equation} \label{eq6.10.1}
\phi_m(n)\ge m\prod_{i=1}^s(1-\frac{1}{p_i})-2^{s-1} \,.
\end{equation}
It is easy to see that $p_i\ge 2i-1$ for all $i\in [2,s]$.
Therefore,
\begin{equation} \label{eq6.10.2}
\prod_{i=1}^s(1-\frac{1}{p_i})\ge
\frac{1}{2}\prod_{i=2}^s\frac{2i-2}{2i-1} \,.
\end{equation}
Since $\frac{2i-2}{2i-1}\ge \frac{2i-3}{2i-2}$ holds for every
$i\in [2,s]$, we obtain that
\[
(\prod_{i=2}^s\frac{2i-2}{2i-1})^2 \ge
\prod_{i=2}^s\frac{2i-2}{2i-1}\prod_{i=2}^s\frac{2i-3}{2i-2}=\prod_{i=1}^{2s-2}\frac{i}{i+1}=\frac{1}{2s-1}.
\]
It follows that
\begin{equation} \label{eq6.10.3}
\prod_{i=2}^s\frac{2i-2}{2i-1}\ge \frac{1}{\sqrt{2s-1}}\,, \quad \text{and hence by (\ref{eq6.10.2})} \qquad
\prod_{i=1}^s(1-\frac{1}{p_i}) \ge \frac{1}{2\sqrt{2s-1}} \,.
\end{equation}
Since $m\ge 2^{s+1}\sqrt{2s-1}$, from  (\ref{eq6.10.3}) we deduce
that $\frac m2 \prod_{i=1}^s(1-\frac{1}{p_i})\ge 2^{s-1}$. It
follows from (\ref{eq6.10.1}) that $\phi_m(n)\ge
m\prod_{i=1}^s(1-\frac{1}{p_i})-2^{s-1}\ge \frac
m2\prod_{i=1}^s(1-\frac{1}{p_i})$.
\end{proof}

\medskip
\begin{lemma} \label{5.10}
Let $n=p_1^{\alpha_1}\cdot \ldots \cdot
p_s^{\alpha_s}$,  where $s\ge 2$, $\alpha_1, \ldots, \alpha_s \in \mathbb N$, and $p_1, \ldots, p_s \in \mathbb P$ are distinct. If $t\le \frac{n+1}{2^{s+1}\sqrt{2s-1}+1}$ then
$\phi_{\lfloor\frac nt\rfloor}(n) \ge
\frac{n}{2.2t}\prod_{i=1}^s(1-\frac{1}{p_i})=\frac{\phi(n)}{2.2t}$.
\end{lemma}

\begin{proof}
From $t\le \frac{n+1}{2^{s+1}\sqrt{2s-1}+1}$ we
obtain that $\lfloor\frac nt\rfloor \ge \frac{n-t+1}{t} \ge
2^{s+1}\sqrt{2s-1}.$  Lemma \ref{5.9} implies that
\begin{equation}\label{eq6.11.1}
\phi_{\lfloor\frac nt\rfloor}(n)\ge \frac{1}{2}\lfloor\frac
nt\rfloor\frac{\phi(n)}{n}\ge \frac{n-t+1}{2t}\frac{\phi(n)}{n} \,.
\end{equation}
Since $s\ge 2$, we infer that $t\le
\frac{n+1}{2^{s+1}\sqrt{2s-1}+1}\le \frac{n+1}{8\sqrt{3}+1}
<\frac{n+11}{11}$. Therefore, $\frac{n-t+1}{2t}\ge \frac{n}{2.2t}$
and the result follows from (\ref{eq6.11.1}).
\end{proof}

\medskip
\begin{lemma} \label{5.11}
Let $n=p_1^{\alpha_1}\cdot \ldots \cdot
p_s^{\alpha_s}$,  where $s\ge 2$, $\alpha_1, \ldots, \alpha_s \in \mathbb N$, and $p_1, \ldots, p_s \in \mathbb P$ are distinct, and let $S$ be a zero-sum sequence over $G$ such that $\ord (g) = n$ for all $g \in \supp (S)$. Then $\min \mathsf L (S)\le \min
\{\frac{|S|}{2},\frac{|S|}{\lfloor\frac{n+1}{2^{s+1}\sqrt{2s-1}+1}\rfloor}+3.3(n-1)(\frac{1}{2}+\log
(\lfloor\frac{n+1}{2^{s+1}\sqrt{2s-1}+1}-1\rfloor))\}$.
\end{lemma}

\begin{proof}
For ease of notation, set $u=\frac{n+1}{2^{s+1}\sqrt{2s-1}+1}$.
It suffices to
show that there is a factorization  $S = U_1 \cdot \ldots
\cdot U_t$, where $U_1, \ldots, U_t \in \mathcal A (G)$
and $t$ is bounded above by the second term in the above set, and
again we construct $U_1, \ldots, U_t$ recursively. For $i \in [1,t]$, let $U_i$
be a minimal zero-sum subsequence  of $S(\prod_{j=1}^{i-1}U_j)^{-1}$ with maximal possible length. We use
Lemma \ref{5.5} and Lemma \ref{5.10} to study $t$.
For every $k\in [2,n]$, let $n_k$ be the number of $U_i$ such that
$|U_i|=k$. For every $m\in [3,n-1]$, the construction of $U_i$ and
Lemma \ref{5.5} imply that
\begin{equation} \label{eq6.12.1}
\sum_{i=2}^{m-1}in_i\le \frac{\phi(n)(n-1)}{\phi_{\lfloor\frac
nm\rfloor}(n)}
\end{equation}
If $m\le u$, then by Lemma \ref{5.10}
we have $\phi_{\lfloor\frac nm\rfloor}(n) \ge
\frac{\phi(n)}{2.2m}$. It follows from (\ref{eq6.12.1}) that
\[
\sum_{i=2}^{m-1}in_i \le  2.2m(n-1)\le 3.3(m-1)(n-1) \,.
\]
Therefore, for every $m\in [2,u-1]$, we have
\begin{equation} \label{eq6.12.2}
\sum_{i=2}^min_i \le 3.3(n-1)m \,.
\end{equation}
Applying Lemma \ref{sum} we deduce that
\[
\sum_{i=2}^mn_i\le 3.3(1+\frac{1}{3}+\frac{1}{4}+\ldots
+\frac{1}{m})(n-1)
\]
holds for every $m\in [2,u-1]$.
Especially,
\[
\sum_{i=2}^{\lfloor u-1\rfloor}n_i\le
3.3(n-1)(\frac{1}{2}+\sum_{i=2}^{\lfloor u-1\rfloor}\frac{1}{i})\le
3.3(n-1)(\frac{1}{2}+\log
(\lfloor u-1\rfloor)) \,.
\]
Since  $\sum_{i\ge
\lfloor u \rfloor}in_i\le |S|$ ,
$\sum_{i\ge \lfloor u \rfloor}n_i\le
\frac{|S|}{\lfloor u\rfloor}$. Hence,
\[
t = \sum_{i=2}^nn_i\le
\frac{|S|}{\lfloor u \rfloor}+3.3(n-1)(\frac{1}{2}+\log
(\lfloor u-1\rfloor))\,. \qedhere
\]
\end{proof}

\medskip
\begin{lemma} \label{5.12}
Let $n=p_1^{\alpha_1}\cdot \ldots \cdot
p_s^{\alpha_s}$, where $s\ge 2$, $\alpha_1, \ldots, \alpha_s \in \mathbb N$,  and $p_1, \ldots, p_s \in \mathbb P$ are distinct. If $n\ge 4376$, then
$\lfloor\frac{n+1}{2^{s+1}\sqrt{2s-1}+1}\rfloor \ge
\frac{n}{2^{s+1}\sqrt{2s}}.$
\end{lemma}

\begin{proof}
Clearly, it suffices to prove that
\[
\frac{n+1}{2^{s+1}\sqrt{2s-1}+1}\ge \frac{n}{2^{s+1}\sqrt{2s}}+1
\]
which will follow from
\begin{equation} \label{eq6.13.1}
(\frac{2^{s+1}}{\sqrt{2s}+\sqrt{2s-1}}-1)n\ge
2^{2s+2}\sqrt{2s(2s-1)}.
\end{equation}
By a straightforward computation we get that (\ref{eq6.13.1}) holds
for  ($s=2$ and $n\ge 250$), for
($s=3$ and $n\ge 656$), for
($s=4$ and $n\ge 1707$), and for
($s=5$ and $n\ge 4376$).
If $s=6$, then $n\ge 2\times 3\times 5\times 7\times 11\times
13=30030$. Again by a straightforward computation we get that
(\ref{eq6.13.1}) holds. Now we proceed by induction. Assume that (\ref{eq6.13.1}) holds for some
$s \ge 6$. Then it holds for $s+1$ because
\[
\begin{array}{ll} (\frac{2^{s+2}}{\sqrt{2(s+1)}+\sqrt{2s+1}}-1)n &\ge
(\frac{2^{s+1}}{\sqrt{2s}+\sqrt{2s-1}}-1)\frac{n}{p_{s+1}^{\alpha_{s+1}}}p_{s+1}^{\alpha_{s+1}}
\\ &\ge
(2^{2s+2}\sqrt{2s(2s-1)})p_{s+1}^{\alpha_{s+1}}
\\ &\ge (2^{2s+2}\sqrt{2s(2s-1)})\times 17 \\ &\ge
2^{2s+4}\sqrt{2(s+1)(2s+1)}. \qquad \qquad \qquad \qquad \qquad \qquad \qquad \qquad \qquad \qquad \quad \qedhere \end{array}
\]
\end{proof}

\begin{lemma} \label{5.13}
Let $n=p_1^{\alpha_1}\cdot \ldots \cdot
p_r^{\alpha_r}$,  where $r\ge 2$, $\alpha_1, \ldots, \alpha_r \in \mathbb N$, and $p_1, \ldots, p_r \in \mathbb P$ are distinct, and let $S$ be a zero-sum sequence over
$G^{\bullet}$. For every divisor $d>1$ of $n$, let $N_d$ denote
the number of the terms of $S$ which have order $d$.
Then $\min \mathsf L (S) \le \min \{\frac{|S|}{2},
\frac{4.3}{2}\sum_{1<d|n}(d-1)+\sum_{1<d|n,d\le
4375}\frac{N_d}{2}+\sum_{4376\le
d|n}\frac{2^{\omega(d)+1}\sqrt{2\omega(d)}N_d}{d}+3.3\sum_{4376\le
d|n}(d-1) \log
(\lfloor\frac{d+1}{2^{\omega(d)+1}\sqrt{2\omega(d)-1}+1}-1\rfloor)\}$.
\end{lemma}

\begin{proof}
It suffices to show that $\min \mathsf L (S)$ is bounded above by the second  term in the above set.
For every $1<d|n$, let $S_d$ denote the subsequence of $S$
consisting of all terms with order $d$, let $T_d$ be a
zero-sum subsequence of $S_d$ with maximal possible length, and set $T_d'=S_dT_d^{-1}$.
Therefore
\[
S=\prod_{1<d|n}T_d'\prod_{1<d|n}T_d \quad \text{and} \quad \prod_{1<d|n}T_d' \quad \text{ has sum zero}.
\]
By the maximality of $T_d$ we
infer that
$
|T_d'|\le d-1
$
for every $1<d|n$, and hence,
\begin{equation}\label{eq6.14.1}
|\prod_{1<d|n}T_d'|\le \sum_{1<d|n}(d-1).
\end{equation}
Therefore,
\[
\min \mathsf L (S)\le \min \mathsf L
(\prod_{1<d|n}T_d')+\sum_{1<d|n}\min \mathsf L (T_d)\le
\frac{|\prod_{1<d|n}T_d'|}{2}+\sum_{1<d|n}\min \mathsf L (T_d).
\]
It follows from (\ref{eq6.14.1}) that
\begin{equation} \label{eq6.14.2}
\min \mathsf L (S)\le \frac{\sum_{1<d|n}(d-1)}{2}+\sum_{1<d|n}\min
\mathsf L (T_d).
\end{equation}
If $\omega(d)\ge 2$ then by Lemma \ref{5.11} and Lemma \ref{5.12}
we obtain that
\begin{equation} \label{eq6.14.3}
\min \mathsf L (T_d)\le
\frac{2^{\omega(d)+1}\sqrt{2\omega(d)}N_d}{d}+3.3(d-1)(\frac{1}{2}+
\log
(\lfloor\frac{d+1}{2^{\omega(d)+1}\sqrt{2\omega(d)-1}+1}-1\rfloor))
\end{equation}
holds for every $4376\le d |n$.
By Lemma \ref{5.6}.2 and Lemma \ref{5.7} we obtain that
(\ref{eq6.14.3}) is true for all $1<d|n$ with $\omega(d)=1$.
It follows from
(\ref{eq6.14.2}) that
\[
\begin{array}{ll} & \min \mathsf L (S) \le
\frac{\sum_{1<d|n}d-1}{2}+\sum_{1<d|n}\min
\mathsf L (T_d) \\
& \le \frac{\sum_{1<d|n}d-1}{2}+\sum_{1<d|n, d\le
4375}\frac{N_d}{2}+\sum_{4376\le
d|n}(\frac{2^{\omega(d)+1}\sqrt{2\omega(d)}N_d}{d}+3.3(d-1)
(\frac{1}{2}+\log
(\lfloor\frac{d+1}{2^{\omega(d)+1}\sqrt{2\omega(d)-1}+1}-1\rfloor))
\\ & \le \frac{4.3}{2}\sum_{1<d|n}(d-1)+\sum_{1<d|n,d\le
4375}\frac{N_d}{2}+\sum_{4376\le
d|n}\frac{2^{\omega(d)+1}\sqrt{2\omega(d)}N_d}{d}
\\ &+3.3\sum_{4376\le d|n}(d-1) \log
(\lfloor\frac{d+1}{2^{\omega(d)+1}\sqrt{2\omega(d)-1}+1}-1\rfloor) \,. \quad \qquad \qquad  \qquad \qquad \qquad \qquad \qquad \quad \qquad \qquad \qedhere
\end{array}
\]
\end{proof}

\begin{proof}[Proof of Theorem \ref{5.2}]
We may suppose that $H$ is reduced and that $H \hookrightarrow \mathcal F (P)$ is a divisor theory with class group $G$. Let $u, v_1, \ldots, v_m, u_2, \ldots, u_{\ell} \in \mathcal A (H)$ be such that $u \t v_1 \cdot \ldots \cdot v_m$, but $u$ divides no proper subproduct, that $v_1 \cdot \ldots \cdot v_m = u u_2 \cdot \ldots \cdot u_{\ell}$, and that $\max \{\ell, m \} = \mathsf t (H, u) = \mathsf t (H)$. By Lemma \ref{4.1}.2, we have $\mathsf t (H) \ge \mathsf t (G) > n \ge |u| \ge m$, and hence we get $1 + \min \mathsf L (w) = \ell = \mathsf t (H) > n$, where $w = u^{-1}v_1 \cdot \ldots \cdot v_m$.
For $i \in [1,m]$, we set $v_i = s_ia_i$ with $a_i, s_i \in \mathcal F (P) \setminus \{1\}$ and $A_i = \boldsymbol \beta (a_i)$. We set $S = A_1 \cdot \ldots \cdot A_m$, and observe that $S = \boldsymbol \beta (w)$ and that $|S| \le m (n-1) \le n(n-1)$. We have to show that $1 + \min \mathsf L (S)$ is bounded above by the terms given in the statement of the theorem.

\medskip
\noindent CASE 1: \, $n = p \in \mathbb P$.

This follows from  Lemma \ref{5.6}.2.

\medskip
\noindent CASE 2: \, $n = p^{\alpha}$, where  $p \in \mathbb P$ and  $\alpha \ge 2$.

For every divisor $d>1$ of $n$, let $N_d$
denote the number of the terms of $S$ which has order $d$.  Since
$A_i$ is zero-sum free, we infer that $A_i$ has at
most $d-1$ terms which have order $d$, and hence
$
N_d \le m(d-1) \le n(d-1)$. Thus it follows from Lemma \ref{5.8} that

\[
\begin{array}{ll}  \min \mathsf L (S) & \le -2\alpha+\frac{2p^{\alpha+1}}{p-1}+\sum_{i=1}^{\alpha}
\frac{2N_{p^i}}{p^i+1}+
3\sum_{i=1}^{\alpha}(p^i-1)\log\frac{p^i}{2}\\ & \le
-2\alpha+\frac{2p^{\alpha+1}}{p-1}+\sum_{i=1}^{\alpha}\frac{2n(p^i-1)}{p^i+1}+
3\sum_{i=1}^{\alpha}(p^i-1)\log\frac{p^i}{2} \\ & \le
-2\alpha+\frac{2p^{\alpha+1}}{p-1}+2\alpha
n+3\sum_{i=1}^{\alpha}(p^i-1)\log\frac{p^i}{2}. \end{array}
\]

\medskip
\noindent CASE 3: \, $n = p_1^{\alpha_1} \cdot \ldots \cdot p_r^{\alpha_r}$, where $r \ge 2$, $\alpha_1, \ldots, \alpha_r \in \N$, and  $p_1, \ldots, p_r \in \mathbb P$ are distinct.

For every divisor $d>1$ of $n$, let $N_d$ denote the number of the
terms of $S$ which have order $d$.  Since $A_i$ is
zero-sum free, we infer that $A_i$ has at most $d-1$
terms whose order divide $d$. Therefore,
$N_d\le m(d-1) \le n(d-1)$ for each $1<d|n$. Now the result
follows from Lemma \ref{5.13}.
\end{proof}

\bigskip
\noindent
{\bf Acknowledgement.} We thank the referees for reading the paper carefully. They have provided a list of helpful comments which, among others, led to a more general version of Lemma 4.3.

%%%%%%%%%%%%%%%%%%%%%%%%%%%%%%%%%%%%%%%%%%%%%%%%%%%%%%%%%%%%%%%%%%%%%%%%%
%% BIBLIOGRAPHY  %%%%%%%%%%%%%%%%%%%%%%%%%%%%%%%%%%%%%%%%%%%%%%%%%%%%%%%%
%%%%%%%%%%%%%%%%%%%%%%%%%%%%%%%%%%%%%%%%%%%%%%%%%%%%%%%%%%%%%%%%%%%%%%%%%

\providecommand{\bysame}{\leavevmode\hbox to3em{\hrulefill}\thinspace}
\providecommand{\MR}{\relax\ifhmode\unskip\space\fi MR }
\providecommand{\MRhref}[2]{%
  \href{http://www.ams.org/mathscinet-getitem?mr=#1}{#2}
}
\providecommand{\href}[2]{#2}

\end{document}